\documentclass[12pt]{article}

\usepackage{amsmath}
\usepackage{amsthm}
\usepackage{amsfonts}

\usepackage{enumitem}
\setlist[enumerate]{label={\upshape(\alph*)}}

\usepackage[numbers,sort&compress]{natbib}

\usepackage{setspace}

\usepackage{color}

\newcommand{\kap}{\overline{\kappa}}
\newcommand{\lamb}{\overline{\lambda}}

\usepackage{hyperref}

\usepackage{tikz}
\usetikzlibrary{calc}
\tikzstyle{vertex}=[circle, draw, inner sep=0pt, minimum size=4pt,fill=black]
\newcommand{\vertex}{\node[vertex]}
\tikzstyle{hollowvertex}=[circle, draw, inner sep=0pt, minimum size=4pt]
\newcommand{\hollowvertex}{\node[hollowvertex]}
\tikzstyle{namedvertex}=[circle, draw, inner sep=1pt, minimum size=12pt]
\newcommand{\namedvertex}{\node[namedvertex]}
\tikzstyle{phantomvertex}=[circle, draw, inner sep=0pt, minimum size=4pt,color=white]
\newcommand{\phantomvertex}{\node[phantomvertex]}

\newtheorem{theorem}{Theorem}[section]

\newtheorem{lemma}[theorem]{Lemma}
\newtheorem{corollary}[theorem]{Corollary}
\newtheorem{observation}[theorem]{Observation}

\theoremstyle{definition}
\newtheorem{definition}[theorem]{Definition}
\newtheorem{conjecture}[theorem]{Conjecture}

\theoremstyle{remark}

\begin{document}

\title{Average connectivity of minimally 2-connected graphs and average edge-connectivity of minimally 2-edge-connected graphs}
\author{
Roc\'{i}o M. Casablanca\\
Universidad de Sevilla\\
Av. Reina Mercedes 2, 41012 Sevilla, Spain\\
\small \href{mailto:rociomc@us.es}{rociomc@us.es}\\
Lucas Mol and Ortrud R. Oellermann\thanks{Supported by an NSERC Grant CANADA, Grant number RGPIN-2016-05237}\\
University of Winnipeg\\
515 Portage Ave.\ Winnipeg, MB, Canada R3B 2E9\\
\small \href{mailto:l.mol@uwinnipeg.ca}{l.mol@uwinnipeg.ca},
 \href{mailto:o.oellermann@uwinnipeg.ca}{o.oellermann@uwinnipeg.ca}\\}

\date{}
\maketitle

\begin{abstract}
Let $G$ be a (multi)graph of order $n$ and let $u,v$ be vertices of $G$. The maximum number of  internally disjoint $u$--$v$ paths in $G$ is denoted by $\kappa_G(u,v)$, and the maximum number of edge-disjoint $u$--$v$ paths in $G$ is denoted by $\lambda_G (u,v)$. The \emph{average connectivity} of $G$ is defined by
\[
\overline{\kappa}(G)=\displaystyle\sum_{\{u,v\}\subseteq V(G)} \kappa_G(u,v)/\tbinom{n}{2},
\]
and the \emph{average edge-connectivity} of $G$ is defined by
\[
\overline{\lambda}(G)=\displaystyle\sum_{\{u,v\}\subseteq V(G)} \lambda_G(u,v)/\tbinom{n}{2}.
\]
A graph $G$ is called \emph{ideally connected} if $\kappa_G(u,v)=\min\{\deg(u),\deg(v)\}$ for all pairs of vertices $\{u,v\}$ of $G$.

We prove that every minimally $2$-connected graph of order $n$ with largest average connectivity is bipartite, with the set of vertices of degree $2$ and the set of vertices of degree at least $3$ being the partite sets.  We use this structure to prove that $\kap(G)<\tfrac{9}{4}$ for any minimally $2$-connected graph $G$.  This bound is asymptotically tight, and we prove that every extremal graph of order $n$ is obtained from some ideally connected nearly regular graph on roughly $n/4$ vertices and $3n/4$ edges by subdividing every edge.    We also prove that $ \lamb(G)<\tfrac{9}{4}$
for any minimally $2$-edge-connected graph $G$, and provide a similar characterization of the extremal graphs.
\end{abstract}

\section{Introduction}\label{Intro}

Throughout, we allow graphs to have multiple edges.  A graph with no multiple edges is called a \emph{simple graph}.  A $u$--$v$ path in a graph $G$ is an alternating sequence of vertices and edges
\[
v_0e_1v_1e_2v_2\dots e_kv_k
\]
in which all vertices and edges are distinct, $v_0=u$, $v_k=v$, and edge $e_i$ has endvertices $v_{i-1}$ and $v_i$ for all $i\in\{1,\dots,k\}$.  If $G$ is a simple graph, then a path can be described by listing only its vertices.  A set of $u$--$v$ paths $\mathcal{P}$ is called \emph{internally disjoint} if no two paths in $\mathcal{P}$ have an internal vertex (i.e., a vertex other than $u$ or $v$) or an edge in common (which can only occur if $u$ and $v$ are adjacent), and is called \emph{edge-disjoint} if no two paths in $\mathcal{P}$ have an edge in common. The {\em distance} between two vertices $u$ and $v$ in $G$, denoted $d_G(u,v)$ or $d(u,v)$ if $G$ is understood, is the length of a shortest $u$--$v$ path.

Let $G$ be a non-trivial graph.  The {\em connectivity} of $G$, denoted  $\kappa(G)$, is the smallest number of vertices whose removal disconnects $G$ or produces the trivial graph.   The \emph{edge-connectivity} of $G$, denoted $\lambda(G)$, is the smallest number of edges whose removal disconnects $G$.  For $k\geq 1$, a graph $G$ is \emph{$k$-connected} (or \emph{$k$-edge-connected}) if it has connectivity (edge-connectivity, respectively) at least $k$.

Following~\cite{BeinekeOellermannPippert2002} we define, for a pair $u,v$ of vertices of $G$, the \emph{connectivity} between $u$ and $v$ in $G$, denoted $\kappa_G(u,v)$, to be the maximum number of internally disjoint paths between $u$ and $v$.  By a well-known theorem of Menger~\cite{Menger1927}, when $u$ and $v$ are non-adjacent, this matches the familiar alternate definition of the connectivity between $u$ and $v$ as the minimum number of vertices whose removal separates $u$ and $v$.  The definition from~\cite{BeinekeOellermannPippert2002} used here is also well-defined if $u$ and $v$ are adjacent.  Analogously, the \emph{edge-connectivity} between $u$ and $v$ in $G$, denoted $\lambda_G(u,v)$, is the maximum number of edge-disjoint $u$--$v$ paths in $G$.  Again, by an alternate version of Menger's theorem, this matches the familiar alternate definition of the edge-connectivity between $u$ and $v$ as the minimum number of edges whose removal separates $u$ and $v$.  When $G$ is clear from context, we use $\kappa(u,v)$ and $\lambda(u,v)$ instead of $\kappa_G(u,v)$ and $\lambda_G(u,v)$, respectively.  Whitney~\cite{Whitney1932} showed that if $G$ is a graph, then $\kappa(G) = \min\{\kappa(u,v)\, |\, u,v \in V(G)\}$. Similarly $\lambda(G) =\{\lambda(u,v)\, |\,  u,v \in V(G)\}$.  Thus, the connectivity and edge-connectivity of a graph are worst-case measures.

A more refined measure of the overall level of connectedness of a graph, introduced in \cite{BeinekeOellermannPippert2002}, is based on the average values of the `local connectivities' between all pairs of vertices.  The {\em average connectivity} of a graph $G$ of order $n$, denoted  $\overline{\kappa}(G)$, is the average of the connectivities over all pairs of distinct vertices of $G$.  That is,
\[
\overline{\kappa}(G)= \sum_{\{u,v \}\subseteq V(G)}\kappa(u,v)/\tbinom{n}{2}.
\]
The \emph{total connectivity} of $G$, denoted $K(G)$, is the sum of the connectivities over all pairs of distinct vertices of $G$, i.e., $K(G)=\binom{n}{2}\kap(G)$.

Analogously, the \emph{average edge-connectivity} of $G$, denoted $\overline{\lambda}(G)$, is the average of the edge-connectivities over all pairs of distinct vertices of $G$.  That is,
\[
\overline{\lambda}(G)=\sum_{\{u,v\}\subseteq V(G)}\lambda_G(u,v)/\tbinom{n}{2}.
\]
The \emph{total edge-connectivity} of $G$, denoted $\Lambda(G)$, is the sum of the edge-connectivities over all pairs of distinct vertices of $G$, i.e., $\Lambda(G)=\binom{n}{2}\overline{\lambda}(G)$.

Let $u$ and $v$ be distinct vertices of a graph $G$.  It is well-known (see~\cite[Section 5]{OellermannChapter2013}) that
\[
\kappa(u,v)\leq \lambda(u,v)\leq \min\{\deg(u),\deg(v)\}.
\]
If $\kappa(u,v)=\min\{\deg(u),\deg(v)\}$ for all pairs of distinct vertices $u$ and $v$ in $G$, then we say that $G$ is \emph{ideally connected}.  If $\lambda(u,v)=\min\{\deg(u),\deg(v)\}$ for all pairs of distinct vertices $u$ and $v$ in $G$, then we say that $G$ is \emph{ideally edge-connected}.  Evidently, if $G$ is ideally connected, then it must also be ideally edge-connected.

Much work has been done on bounding the average connectivity in terms of various graph parameters, including order and size~\cite{BeinekeOellermannPippert2002}, average degree~\cite{DankelmannOellermann2003}, and matching number~\cite{KimO2013}.  Bounds have also been achieved on the average connectivity of graphs belonging to particular families, including planar and outerplanar graphs~\cite{DankelmannOellermann2003}, Cartesian product graphs~\cite{DankelmannOellermann2003}, strong product graphs~\cite{Abajo2013}, and regular graphs~\cite{KimO2013}.  Average connectivity has also proven to be a useful measure for real-world networks, including street networks~\cite{Boeing2017} and communication networks~\cite{Rak2015}.

In this article, we demonstrate sharp bounds on the average connectivity of minimally $2$-connected graphs and the average edge-connectivity of minimally $2$-edge-connected graphs.  For $k\geq 1$, a graph $G$ is called {\em minimally $k$-connected} if $\kappa(G)=k$ and for every edge $e$ of $G$, $\kappa(G-e) < k$.  Analogously, $G$ is called \emph{minimally $k$-edge-connected} if $\lambda(G)=k$ and for every edge $e$ of $G$, $\kappa(G-e)<k$.  A graph $G$ with $\kappa(G) = \overline{\kappa}(G) =k$ is called a \emph{uniformly $k$-connected graph}.  It was observed in \cite{BeinekeOellermannPippert2002} that uniformly $k$-connected graphs are minimally $k$-connected.  It is obvious that every minimally $1$-connected graph (i.e., tree) is uniformly $1$-connected.  However, for $k\geq 2$, minimally $k$-connected graphs need not be uniformly $k$-connected, as can be seen by considering the graphs $K_{k,n-k}$ for $n > 2k \ge 4$.  So if $k\geq 2$, it is natural to ask by how much the average connectivity of a minimally $k$-connected graph can exceed $k$.  Similarly, by how much can the average edge-connectivity of a minimally $k$-edge-connected graph exceed $k$? In this article, we answer both of these questions in the case where $k=2$.

We show that
\[
2\leq \overline{\kappa}(G)< \tfrac{9}{4}
\]
for every minimally $2$-connected graph $G$.  The lower bound is readily seen to be attained if and only if $G$ is a cycle.  We prove the upper bound in Section~\ref{Vertex}. We say that $G$ is an \emph{optimal} minimally $2$-connected graph of order $n$ if $G$ has maximum average connectivity among all such graphs.  We prove that any optimal minimally $2$-connected graph of order at least $5$ must be bipartite, with the set of vertices of degree $2$ and the set of vertices of degree at least $3$ being the partite sets.  
 More specifically it is shown that every minimally $2$-connected graph of order $n$ having maximum average connectivity are those obtained from some ideally connected nearly regular graph on roughly $n/4$ vertices and $3n/4$ edges by subdividing every edge. This result demonstrates that the above bound of $9/4$ on $\overline{\kappa}(G)$ is asymptotically tight.  It can be deduced, from this characterization, that the optimal minimally $2$-connected graphs are ideally connected but not all ideally minimally $2$-connected graphs are optimal.

We also show that
\[
2\leq \overline{\lambda}(G)< \tfrac{9}{4}
\]
for any minimally $2$-edge-connected graph $G$.  Once again, the lower bound is readily seen to be attained if and only if $G$ is a cycle.  We prove the upper bound in Section~\ref{Edge}, where we study the structure of minimally $2$-edge-connected graphs of order $n$ with maximum average edge-connectivity (which we call \emph{edge-optimal} minimally $2$-edge-connected graphs).  We obtain structural results on edge-optimal minimally $2$-edge-connected graphs similar to those obtained for optimal minimally $2$-connected graphs, though the proofs are quite different.  This culminates in the same upper bound as for the vertex version, and an analogous characterization of the edge-optimal minimally $2$-edge-connected graphs.

Before we proceed, we introduce some notation that is used throughout.  For vertices $u$ and $v$ in a graph $G$ we use $u \sim_G v$ to indicate that $u$ is adjacent with $v$ and $u \not\sim_G v$ to indicate that $u$ is not adjacent with $v$. The subscript is omitted if $G$ is clear from context.  If $P$ is a path, then $\overleftarrow{P}$ is the path obtained by reversing the order of the vertices and edges in $P$. Let $u$ and $v$ be vertices of $P$ where $u$ precedes $v$ on $P$. Then $P[u,v]$ denotes the $u$--$v$ subpath of $P$.  If $P_1$ is a path ending in $u$ and $P_2$ is a path beginning in $u$, then we let $P_1\odot P_2$ denote the concatenation of $P_1$ and $P_2$, with $u$ written only once.  Note that $P_1\odot P_2$ is a path if and only if $P_1$ and $P_2$ have no vertices in common apart from $u$.

\section{Average connectivity of minimally 2-connected graphs}\label{Vertex}

In this section, we obtain results about the structure of optimal minimally $2$-connected graphs, and use this to prove a sharp upper bound on the average connectivity of minimally $2$-connected graphs. It is easy to see that minimally $2$-connected graphs must be simple graphs. So throughout this section, we denote paths by listing only the vertices.

We begin with some background material on minimally $k$-connected graphs.  An edge $e$ of a $k$-connected graph $G$ is called $k$-{\em essential} if $\kappa(G-e) < \kappa(G)$.  Thus, a minimally $k$-connected graph is one in which every edge is $k$-essential.  Mader \cite{Mader1972} established the following structure theorem for the $k$-essential edges in a $k$-connected graph.

\begin{theorem}[\cite{Mader1972}]
If $G$ is a $k$-connected graph and if $C$ is a cycle of $G$ in which every edge is $k$-essential, then some vertex of $C$ has degree $k$ in $G$.
\end{theorem}

The following structural results for minimally $k$-connected graphs are an immediate consequence.

\begin{corollary}[\cite{Mader1972}]\label{mader2}
If $G$ is a minimally $k$-connected graph, then $G$ has a vertex of degree $k$.
\end{corollary}

\begin{corollary}[\cite{Mader1972}]\label{mader}
Let $G$ be a minimally $k$-connected graph and $F$ the subgraph of $G$ induced by the vertices of degree exceeding $k$. Then $F$ is a forest.
\end{corollary}

Minimally $2$-connected graphs were characterized independently in \cite{Dirac1967,Plummer1968}. A cycle $C$ of a graph $G$ is said to have a {\em chord} if there is an edge of $G$ that joins a pair of non-adjacent vertices from $C$.  The following characterization of Plummer \cite{Plummer1968} is used frequently throughout this section.

\begin{theorem}[{\cite[Corollary 1a]{Plummer1968}}] \label{plummer}
A $2$-connected graph $G$ is minimally $2$-connected if and only if no cycle of $G$ has a chord.
\end{theorem}

\subsection{Structural properties of optimal minimally 2-connected graphs}

Let $G$ be a minimally $2$-connected graph, and let $F$ be the subgraph of $G$ induced by the set of vertices of degree exceeding $2$.  By Corollary~\ref{mader}, $F$ is a forest.  We begin by proving that if $u$ and $v$ are in the same component of $F$, then $\kappa_G(u,v)=2$. The special case where two vertices are adjacent was observed in {\cite[Lemma 4.2]{Bollobas1978}}.  This explains why we might expect the set of vertices of degree exceeding $2$ to be independent in an optimal minimally $2$-connected graph.

\begin{theorem} \label{con_between_pairs_large_degree}
Let $G$ be a minimally $2$-connected  graph that is not a cycle, and let $F$ be the subgraph of $G$ induced by the set of vertices of $G$ of degree exceeding $2$.  If $u$ and $v$ are distinct vertices of $G$ that belong to the same component of $F$, then $\kappa(u,v)=2$.
\end{theorem}
\begin{proof}   By Corollary \ref{mader}, $F$ is a forest. Since $G$ is minimally $2$-connected, $\kappa(u,v) \ge 2$. So it remains to be shown that $\kappa(u,v) \le 2$.  Assume, to the contrary, that $\kappa(u,v)\geq 3$.

Suppose first that $uv$ is an edge of $F$.  Then there exist at least two internally disjoint $u$--$v$ paths $P_1$ and $P_2$ in $G$, each of length at least 2. So $uv$ is a chord of the cycle produced by $P_1$ and $P_2$, contrary to Theorem \ref{plummer}.

Suppose now that $d_F(u,v)=2$. Let $w$ be the common neighbour of $u$ and $v$ in $F$.  Let $P$ be the path $uwv$, and let $P_1$ and $P_2$ be two $u$--$v$ paths internally disjoint from $P$ and one another.   Since $w$ is in $F$, $\deg_G(w) \ge 3$.  Say $w$ is adjacent with $x\neq u,v$.  If $x$ is in $P_1$, then the edge $wx$ is a chord of the cycle formed by $P$ and $P_1$, contrary to Theorem~\ref{plummer}.  Similarly, $x$ is not in $P_2$.  Since $G$ is $2$-connected, there is an $x$--$u$ path $Q$ that does not contain $w$. Let $z$ be the first vertex of $Q$ that lies on either $P_1$ or $P_2$, say $P_1$ (note that possibly $z=u$). Then the paths $P_2$, $vwx$, $\overleftarrow{Q}[x,z]$, and ${P_1}[z,u]$ make up a cycle with the chord $uw$, contrary to Theorem \ref{plummer}.

We may now assume that $d_F(u,v) \ge 3$. We choose  $u$ and $v$ in such a way that $\kappa(u,v)\geq 3$ and $d_F(u,v)$ is as small as possible. Thus, if $a$ and $b$ are two vertices in the same component of $F$ with $d_F(a,b) < d_F(u,v)$, then $\kappa(a,b) =2$. Let $d_F(u,v)=k\geq 3$ and let $P:(u=)u_0u_1 \dots u_k(=v)$ be the $u$--$v$ path in $F$. Let $\mathcal{P}$ be a collection of $\kappa(u,v)$ pairwise internally disjoint $u$--$v$ paths in $G$.

\medskip

\noindent
{\bf Claim:} $P \in \mathcal{P}$.\\
Assume, to the contrary, that $P \not\in \mathcal{P}$. Let $i$ be the smallest positive integer such that $u_i$ lies on some path, $P_1$ say, of $\mathcal{P}$.  We must have $1\le i<k$, since otherwise $P$ would be internally disjoint from all paths in $\mathcal{P}$, contradicting the maximality of $\mathcal{P}$.

Suppose first that $2 \le i <k$.  Let $P_2 \in \mathcal{P}-\{P_1\}$.  Since $u_k=v$ is on $P_2$, there is a smallest positive integer $j$  such that $u_j$ is on $P_2$. By our choice of $i$, we see that $j >i$.   Then $P[u,u_i]$,  $P_1[u,u_i]$, and $P_2[u,u_j]\odot\overleftarrow{P}[u_j, u_i]$ are three internally disjoint $u$--$u_i$ paths in $G$.  Since $d_F(u,u_i)=i<k=d_F(u,v)$, this contradicts our choice of $u$ and $v$.

So we may assume that $i=1$; that is, $u_1$ is on $P_1$.  Since we are assuming that $P \not\in \mathcal{P}$, we must have $P\neq P_1$, and hence there is a smallest $j$, $1 \le j <k-1$ such that $u_{j+1}$ is not on $P_1$ (if $P_1$ contained $u_0,u_1,\dots,u_{k-1}$ then we could swap $P_1$ for $P$ in $\mathcal{P}$). By Theorem \ref{plummer}, $u_{j+1}$ does not belong to any path of $\mathcal{P} -\{P_1\}$. Since $F$ is a forest, $u_{j+1} \not\sim
u_i$ for $0 \le i \le j-1$. Let $\ell > j+1$ be the first integer such that $u_\ell$ lies on some path in $\mathcal{P}$ (possibly $\ell=k$). If $u_\ell$ is not on $P_1$ (so in particular $u_\ell \ne v$), then $u_\ell$ is on some path, say $P_2$, of $\mathcal{P} - \{P_1\}$.  Let $P_3 \in \mathcal{P} - \{P_1,P_2\}$.  Then $P_1[u,u_j]$,  $P_2[u,u_\ell]\odot\overleftarrow{P}[u_\ell,u_j]$, and $P_3\odot\overleftarrow{P_1}[v,u_j]$ are internally disjoint $u$--$u_j$ paths in $G$. Since $d_F(u,u_j) < d_F(u,v)$, this contradicts our choice of $u$ and $v$.  Thus  $u_\ell$ is on $P_1$.  Then $P[u_j,u_\ell]$, $P_1[u_j,u_\ell]$, and $\overleftarrow{P}[u_j,u_0]\odot P_2\odot \overleftarrow{P}[u_k,u_\ell]$ are three internally disjoint $u_j$--$u_\ell$ paths in $G$.  Since $d_F(u_j,u_\ell) < d_F(u,v)$, this again contradicts our choice of $u$ and $v$.  This completes the proof of our claim. 

\medskip

So $P\in \mathcal{P}$.  Let $P_1$ and $P_2$ be paths in $\mathcal{P}$ distinct from $P$ and one another.  Since $u_1$ is in $F$, it has degree at least $3$ and hence has a neighbour $x$ not on $P$. By Theorem \ref{plummer}, $x$ does not lie on any path of $\mathcal{P}$. Since $G$ is $2$-connected, there is an $x$--$u$ path $Q$ that does not contain $u_1$.  Let $z$ be the first vertex of $Q$ that is on a path in $\mathcal{P}$.  If $z=u$, then $P[u,u_1],$ $\overleftarrow{Q}\odot xu_1$, and $P_1\odot \overleftarrow{P}[v,u_1]$ are three internally disjoint $u$--$u_1$ paths in $G$, so $\kappa_G(u,u_1)\geq 3$.  Otherwise, if $z\neq u$ but $z$ is on $P$, then $z$ is in $F$, and $P[u_1,z]$, $u_1x\odot Q[x,z]$, and $u_1u\odot P_2\odot \overleftarrow{P}[v,z]$ are three internally disjoint $u_1$--$z$ paths in $G$, so $\kappa(u_1,z) \ge 3$.  Finally, if $z$ is not on $P$, then assume without loss of generality that $z$ lies on $P_1$.  Then $P[u,u_1]$, $P_1[u,z]\odot \overleftarrow{Q}\odot xu_1$, and $P_2\odot \overleftarrow{P}[v,u_1]$ are three internally disjoint $u$--$u_1$ paths in $G$, so $\kappa(u,u_1) \ge 3$.  Either way, this contradicts our choice of $u$ and $v$, and this completes the proof.
\end{proof}

We now show that if $G$ is an optimal minimally $2$-connected graph of order at least $5$, then the set of vertices of degree $2$ is independent, and so is the set of vertices of degree exceeding $2$.  This is the key structural result used in the sequel to obtain an upper bound on the average connectivity of minimally $2$-connected graphs.

\begin{theorem} \label{bipartite}
Let  $G$ be an optimal minimally $2$-connected graph of order $n \ge 5$.  Then $G$ is bipartite with partite sets the set of vertices of degree $2$ and the set of vertices of degree exceeding $2$.
\end{theorem}

\begin{proof} Since $n \ge 5$ and $K_{2,n-2}$ is a minimally $2$-connected graph with average connectivity exceeding $2$, $G$ is not a cycle. So $G$ has at least two vertices of degree exceeding $2$. We show first that the vertices of degree $2$ form an independent set. If this is not the case, then there exist vertices $u$ and $v$ of degree exceeding $2$ and a $u$--$v$ path  $P: (u=)u_0u_1 \ldots u_k(=v)$, such that $k \ge 3$ and $\deg_G(u_i)=2$ for $1 \le i <k$. Delete the edges of $P$ from $G$ and add the edges $uu_i$ and $u_iv$ for $1 \le i \le k-1$. Let $G'$ be the resulting graph. Then $G'$ has order $n$ and it is readily checked that $G'$ is minimally $2$-connected. Moreover, the total connectivity of $G'$ exceeds the total connectivity of $G$ by $k-2$ since $\kappa_{G'}(u,v)=\kappa_G(u,v) + k-2$, and for all pairs $x,y$ of vertices of $G$ where $\{x,y\}\ne\{u,v\}$ we have $\kappa_{G'}(x,y)=\kappa_{G}(x,y)$.

\medskip

It remains to show that the set of vertices of degree exceeding $2$ is independent.  Suppose, towards a contradiction, that $u$ and $v$ are adjacent vertices of degree at least $3$ in $G$.  Since $G$ is minimally $2$-connected, $G-uv$ has a cut-vertex, say $x$.  Since $G-x$ is connected, it follows that $uv$ is a bridge of $G-x$. So $G-uv-x$ has exactly two components $G_1$ and $G_2$, say, where $G_1$ contains $u$ and $G_2$ contains $v$.  Let $G_i'$ be the subgraph of $G$ induced by $V(G_i) \cup \{x\}$ for $i=1,2$ (see Figure~\ref{VertexStructurePicture}).

\begin{figure}[htb]
\centering{
\begin{tikzpicture}
\path[use as bounding box] (-2, -1.7) rectangle (2, 2.3);
\pgftransformrotate{-20}
\pgftransformxshift{-0.5cm}
\coordinate (A) at (0,-2);
\coordinate (B) at (0,-0.8);
\coordinate (C) at (0,0.8);
\coordinate (D) at (0,2);
\draw (A) to[out=210,in=150,distance=2.5cm] (D);
\draw (D) to[out=-30,in=30,distance=0.5cm] (C);
\draw (C) to[out=210,in=150] (B);
\draw (B) to[out=-30,in=30,distance=0.5cm] (A);
\pgftransformreset
\pgftransformrotate{20}
\pgftransformxshift{0.5cm}
\coordinate (A) at (0,-2);
\coordinate (B) at (0,-0.8);
\coordinate (C) at (0,0.8);
\coordinate (D) at (0,2);
\draw[dashed] (A) to[out=-30,in=30,distance=2.5cm] (D);
\draw[dashed] (D) to[out=210,in=150,distance=0.5cm] (C);
\draw[dashed] (C) to[out=-30,in=30] (B);
\draw[dashed] (B) to[out=210,in=150,distance=0.5cm] (A);
\pgftransformreset
\namedvertex (x) at (0,1.7) {\tiny $x$};
\namedvertex (u) at (-1.1,-1.1) {\tiny $u$};
\namedvertex (v) at (1.1,-1.1) {\tiny $v$};
\phantomvertex (un1) at (-0.75,1.7) {};
\phantomvertex (vn1) at (0.75, 1.9) {};
\phantomvertex (vn3) at (0.75, 1.5) {};
\phantomvertex (u1) at (-1.5,-0.6) {};
\phantomvertex (u2) at (-1.25,-0.5) {};
\phantomvertex (v1) at (1.5,-0.6) {};
\phantomvertex (v2) at (1.25,-0.5) {};
\path
(x) edge (un1)
(x) edge (vn1)
(x) edge (vn3)
(u) edge (v)
(u) edge (u1)
(u) edge (u2)
(v) edge (v2)
(v) edge (v1)
;
\draw (-1.5,0) node {$G'_1$};
\draw (1.5,0) node {$G'_2$};
\end{tikzpicture}
\hspace{1cm}
\begin{tikzpicture}
\path[use as bounding box] (-1.5, -2) rectangle (1.5, 2);
\coordinate (A) at (0,-2);
\coordinate (B) at (0,-0.8);
\coordinate (C) at (0,0.8);
\coordinate (D) at (0,2);
\draw (A) to[out=210,in=150,distance=2.5cm] (D);
\draw (D) to[out=-30,in=30,distance=0.5cm] (C);
\draw (C) to[out=210,in=150] (B);
\draw (B) to[out=-30,in=30,distance=0.5cm] (A);
\draw[dashed] (A) to[out=-30,in=30,distance=2.5cm] (D);
\draw[dashed] (D) to[out=210,in=150,distance=0.5cm] (C);
\draw[dashed] (C) to[out=-30,in=30] (B);
\draw[dashed] (B) to[out=210,in=150,distance=0.5cm] (A);
\namedvertex (x) at (0,1.4) {\tiny $x$};
\namedvertex (w) at (0,-1.4) {\tiny $w$};
\phantomvertex (un1) at (-0.75,1.4) {};
\phantomvertex (vn1) at (0.75, 1.6) {};
\phantomvertex (vn3) at (0.75, 1.2) {};
\phantomvertex (u1) at (-0.75,-1.6) {};
\phantomvertex (u2) at (-0.75,-1.2) {};
\phantomvertex (v1) at (0.75,-1.6) {};
\phantomvertex (v2) at (0.75,-1.2) {};
\path
(x) edge (un1)
(x) edge (vn1)
(x) edge (vn3)
(w) edge (u1)
(w) edge (u2)
(w) edge (v2)
(w) edge (v1)
;
\end{tikzpicture}
\hspace{1cm}
\begin{tikzpicture}
\path[use as bounding box] (-1.5, -2) rectangle (1.5, 2);
\coordinate (A) at (0,-2);
\coordinate (B) at (0,-0.8);
\coordinate (C) at (0,0.8);
\coordinate (D) at (0,2);
\draw (A) to[out=210,in=150,distance=2.5cm] (D);
\draw (D) to[out=-30,in=30,distance=0.5cm] (C);
\draw (C) to[out=210,in=150] (B);
\draw (B) to[out=-30,in=30,distance=0.5cm] (A);
\draw[dashed] (A) to[out=-30,in=30,distance=2.5cm] (D);
\draw[dashed] (D) to[out=210,in=150,distance=0.5cm] (C);
\draw[dashed] (C) to[out=-30,in=30] (B);
\draw[dashed] (B) to[out=210,in=150,distance=0.5cm] (A);
\namedvertex (x) at (0,1.4) {\tiny $x$};
\namedvertex (w) at (0,-1.4) {\tiny $w$};
\phantomvertex (un1) at (-0.75,1.4) {};
\phantomvertex (vn1) at (0.75, 1.6) {};
\phantomvertex (vn3) at (0.75, 1.2) {};
\phantomvertex (u1) at (-0.75,-1.6) {};
\phantomvertex (u2) at (-0.75,-1.2) {};
\phantomvertex (v1) at (0.75,-1.6) {};
\phantomvertex (v2) at (0.75,-1.2) {};
\namedvertex (y) at (0,0) {\tiny $y$};
\path
(x) edge (un1)
(x) edge (vn1)
(x) edge (vn3)
(w) edge (u1)
(w) edge (u2)
(w) edge (v2)
(w) edge (v1)
(w) edge (y)
(y) edge (x)
;
\end{tikzpicture}
}
\caption{A sketch of $G$ (left), $H$ (middle), and $G'$ (right).  Note that $u$, $v$, and $x$ do not necessarily have degree exactly $3$ as drawn.}\label{VertexStructurePicture}
\end{figure}
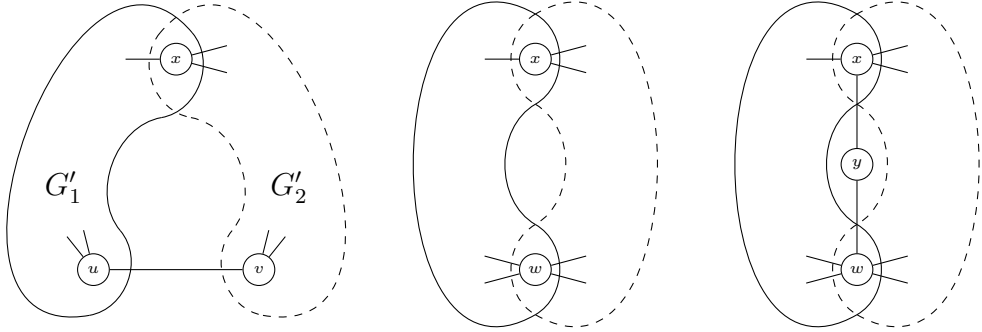

\medskip

\noindent{\bf Fact 1:} $u \not\sim_G x$, and $v \not\sim_G x$.

\smallskip

\noindent
Since $\deg_G(u) \ge 3$, $u$ has a neighbour $a$ in $G-\{v,x\}$. Since $G$ is $2$-connected, there is an $a$--$x$ path $Q_1$ that does not contain $u$, and it must lie in $G'_1$.  Similarly, $v$ has a neighbour $b$ in $G-\{u,x\}$, and there is a $b$--$x$ path $Q_2$ in $G_2'$ that does not contain $v$. So the paths $Q_1$, $Q_2$, and $auvb$ produce a cycle, $C$ say, in which neither $u$ nor $v$ is adjacent with $x$. Since $G$ is minimally $2$-connected, it follows from Theorem \ref{plummer} that $C$ has no chords.  So $u \not\sim x$ and $v \not\sim x$. This completes the proof of Fact 1.

\medskip

Let $H$ be the graph obtained from $G$ by contracting the edge $uv$ to a new vertex labeled $w$.  Let $G'$ be the graph obtained from $H$ by adding vertex $y$ and the edges $xy$ and $yw$ (see Figure~\ref{VertexStructurePicture}).  We prove that $G'$ is a minimally $2$-connected graph of order $n$ with $\kap(G')>\kap(G)$, contradicting the optimality of $G$.

\medskip

\noindent{\bf Fact 2:} $G'$ is $2$-connected. 

\smallskip

\noindent
We show that $H$ is $2$-connected.  Since $G'$ is obtained from $H$ by joining the new vertex $y$ to the vertices $x$ and $w$, it follows by a straightforward argument that $G'$ is also $2$-connected.  To see that $H$ is $2$-connected, we show that every pair of distinct vertices of $H$ lies on a cycle.  First let $a,b\in V(H)-\{w\}=V(G)-\{u,v\}$.  Since $G$ is $2$-connected, there is a cycle $C$ of $G$ containing $a$ and $b$. If $C$ does not contain $u$ or $v$, then $C$ is a cycle of $H$.  If $C$ contains exactly one of $u$ (or $v$), then the cycle obtained from $C$ by replacing $u$ (or $v$, resp.) with $w$ is a cycle of $H$ that contains $a$ and $b$.  So we may assume that $C$ contains both $u$ and $v$.  In this case, $C$ must contain $x$ as well.  By Fact 1, $u \not\sim_G x$ and $v \not\sim_G x$.  So by contracting  the edge $uv$ of $C$ to $w$, we obtain a cycle of $H$ containing $a$ and $b$.  Finally, for any $a\in V(H)-\{w\}$, there is a cycle $C$ of $G$ containing $a$ and the edge $uv$.  Contracting the edge $uv$ of $C$ to $w$ gives a cycle in $H$ containing $a$ and $w$.  So we conclude that $H$, and hence $G'$, is $2$-connected. This completes the proof of Fact 2.

\medskip

\noindent{\bf Fact 3:} $G'$ is minimally $2$-connected.

\smallskip

\noindent
By Fact 2, $G'$ is $2$-connected.  Let $e$ be an edge in $G'$.  We need to show that $G'-e$ has a cut vertex.  First of all, if $e$ is incident to $y$, then either $x$ or $w$ is a cut vertex of $G'-e$.  Next, if $e$ is incident to $w$, then assume without loss of generality that the other endvertex of $e$ is $a\in V(G'_1)$.  Then $au$ is an edge of $G$, and $G-au$ has a cut vertex, say $z$.  If $z\in V(G'_1)$, then $z$ is also a cut vertex in $G'-e$.  Otherwise, if $z\in V(G'_2)$, then every $a$--$u$ path in $G-au$ contains both $z$ and $x$.  In turn, every $a$--$w$ path in $G'-e$ contains $x$, so $x$ is a cut vertex in $G'-e$.

We may now assume that $e$ is not incident with $y$ or $w$.  Without loss of generality, let $e\in E(G'_1)$.  Now $G-e$ has a cut vertex, say $z$.  If $z\in V(G'_1)-\{u\}$, then $z$ is also a cut vertex in $G'-e$.  Otherwise, $z\in V(G_2)\cup\{u\}$.  Let $a$ and $b$ be the endvertices of $e$, and note that every $a$--$b$ path in $G-e$ contains $z$, and hence $u$.  It follows that every $a$--$b$ path in $G'-e$ contains $w$, so $w$ is a cut vertex in $G'-e$. This completes the proof of Fact 3.

\medskip

\noindent{\bf Fact 4:} $\overline{\kappa}(G)<\overline{\kappa}(G')$.

\smallskip

\noindent
We show that $K(G)<K(G')$, from which the statement readily follows.  We demonstrate the following:
\begin{enumerate}[label=(\roman*)]
\item $\kappa_G(u,v)=\kappa_{G'}(w,y)$;
\item $\kappa_G(a,b)\leq \kappa_{G'}(a,b)$ for all $a,b\in V(G)-\{u,v\}$;
\item $\kappa_G(u,z)+\kappa_G(v,z)\leq \kappa_{G'}(w,z)+\kappa_{G'}(y,z)$ for all $z\in V(G)-\{u,v,x\}$; and
\item $\kappa_G(u,x)+\kappa_G(v,x)<\kappa_{G'}(w,x)+\kappa_{G'}(y,x)$.
\end{enumerate}
Summing the left-hand side of (i), (ii), (iii), and (iv) over all possibilities gives $K(G)$, and summing the right-hand side of (i), (ii), (iii), and (iv) over all possibilities gives $K(G')$, so the desired result follows immediately.

For (i), $\kappa_G(u,v)=2$ by Theorem~\ref{con_between_pairs_large_degree}, and since $\deg_{G'}(y)=2$, we have $\kappa_{G'}(w,y) = 2$, by Fact 2.

For (ii), let $a,b\in V(G)-\{u,v\}=V(G')-\{w,y\}$.  If one of $a$ and $b$ belongs to $G_1$ and the other to $G_2$, then $\kappa_{G'}(a,b)=2=\kappa_G(a,b)$.  So assume, without loss of generality, that $a,b\in V(G'_1)-u$. If $\mathcal{P}_{a,b}$ is a collection of $\kappa_G(a,b)$ pairwise internally disjoint $a$--$b$ paths in $G$, then at most one of these paths contains the vertex $u$.  If no member of $\mathcal{P}_{a,b}$ contains $u$, then $\mathcal{P}_{a,b}$ is a collection of $\kappa_G(a,b)$ internally disjoint $a$--$b$ paths in $G'$.  Otherwise, let $P$ be the unique path in $\mathcal{P}_{a,b}$ containing $u$.  If $v$ is also on $P$, then let $P'$ be the path obtained from $P$ by contracting $uv$ to $w$.  Otherwise, if $v$ is not on $P$, then let $P'$ be the path obtained from $P$ by replacing $u$ with $w$.  Then $(\mathcal{P}_{a,b}-P)\cup\{P'\}$ is a collection of $\kappa_G(a,b)$ internally disjoint $a$--$b$ paths in $G'$.  Either way, we conclude that $\kappa_{G}(a,b)\leq \kappa_{G'}(a,b)$.

For (iii), let $z\in V(G)-\{u,v,x\}$.  Assume without loss of generality that $z \in V(G_1)$. Let $\mathcal{P}_{u,z}$ be a family of $\kappa_G(u,z)$ pairwise internally disjoint $u$--$z$ paths in $G$. Any path between $u$ and $z$ that also contains at least one vertex of $G_2$ must necessarily contain both $uv$ and $x$. Thus at most one of the paths in $\mathcal{P}_{u,z}$ contains $v$. If such a $u$--$z$ path $P$ exists, then the path obtained from $P$ by contracting the edge $uv$ to $w$ is a $w$--$z$ path in $G'$. If we replace $u$ by $w$ on all the remaining paths in $\mathcal{P}_{u,z}$, then we obtain a family of $\kappa_G(u,z)$ pairwise internally disjoint $w$--$z$ paths in $G'$. So $\kappa_G(u,z) \le \kappa_{G'}(w,z)$. Since the edge $uv$ and the vertex $x$ separate $z$ and $v$ in $G$, it follows that $\kappa_G(v,z)=2$. Since $\deg_{G'}(y)=2$, we have $\kappa_{G'}(y,z)=2$, by Fact 2. So $\kappa_G(u,z) + \kappa_G(v,z) \le \kappa_{G'}(w,z) + \kappa_{G'}(y,z)$.

For (iv), let $\mathcal{P}_{u,x}$ be a collection of $\kappa_G(u,x)$ pairwise internally disjoint $u$--$x$ paths in $G$.  Exactly one of these paths contains vertices of $G_2$, since such a path necessarily contains the edge $uv$, and there is a $v$--$x$ path in $G'_2$. Let $\mathcal{P}'_{u,x}$ be the collection of all paths in $\mathcal{P}_{u,x}$ whose internal vertices belong to $G_1$. So $|\mathcal{P}'_{u,x}|= \kappa_G(u,x)-1$.  By replacing $u$ with $w$ on every path of $\mathcal{P}'_{u,x}$, we obtain a family $\mathcal{P}''_{u,x}$ of $\kappa_G(u,x)-1$ internally disjoint $w$--$x$ paths of $G'$ whose internal vertices all belong to $G_1$.  By a similar argument, we obtain a family $\mathcal{P}''_{v,x}$ of $\kappa_G(v,x)-1$ internally disjoint $w$--$x$ paths of $G'$ whose internal vertices all belong to $G_2$.  The path $wyx$ is a $w$--$x$ path that is internally disjoint from the paths in $\mathcal{P}''_{u,x} \cup \mathcal{P}''_{v,x}$.  So $\kappa_{G'}(w,x) \geq \kappa_G(u,x)+\kappa_G(v,x)-1$.  Finally, since $\deg_{G'}(y)=2$, we have $\kappa_{G'}(y,x)= 2$, by Fact 2.  Therefore,
\[
\kappa_{G'}(w,x)+ \kappa_{G'}(y,x) \ge \kappa_G(u,x) + \kappa_G(v,x)+1 > \kappa_G(u,x) + \kappa_G(v,x).
\]
This completes the proof of Fact 4 and the theorem.
\end{proof}

We conclude this section by noting that, given a minimally $2$-connected graph $G$ of order $n\geq 5$, for which either the vertices of degree $2$ or the vertices of degree exceeding $2$ are not independent, the proof of Theorem~\ref{bipartite} implicitly describes an algorithm for constructing a minimally $2$-connected graph $G'$ of the same order $n$ with higher average connectivity than $G$. By repeated application of this algorithm we obtain a minimally $2$-connected graph of order $n$ in which the vertices of degree $2$ and those of degree exceeding $2$ are independent. Moreover, the average connectivity of this graph exceeds that of the other graphs that preceded it in the process.

\subsection{An upper bound on the average connectivity of minimally 2-connected graphs}\label{VertexBoundSection}

Using the structural results on optimal minimally $2$-connected graphs obtained in the previous section, we now demonstrate a sharp upper bound on the average connectivity of a minimally $2$-connected graph of order $n$, and characterize the optimal minimally $2$-connected graphs of order $n$, for all $n$ sufficiently large.

We require some terminology.  A graph $G$ is \emph{nearly regular} if the difference between its maximum degree and its minimum degree is at most $1$.  If $G$ is a nearly regular graph of order $n$ and size $m$, then $G$ has degree sequence
\[
\underbrace{d,\dots,d}_{\text{$n-r$ terms}},\underbrace{d+1,\dots,d+1}_{\text{$r$ terms}}
\]
where $d,r\in\mathbb{Z}$ are the unique integers satisfying $2m=dn+r$ and $0\leq r<n$.  We call this sequence a \emph{nearly regular} sequence.

Let $G$ be a graph.  We know that $\kappa(u,v)\leq \min\{\deg(u),\deg(v)\}$ for all pairs of distinct vertices $u$ and $v$ of $G$.  This motivates the following definition.

\begin{definition}
The \emph{potential} of a sequence of positive integers $d_1,d_2,\dots,d_n$ is defined by
\[
P(d_1,d_2,\dots,d_n)=\sum_{1\leq i<j\leq n} \min\{d_i,d_j\}.
\]
For a graph $G$ on $n$ vertices $v_1,v_2,\dots,v_n$, the \emph{potential} of $G$, denoted $P(G)$, is the potential of the degree sequence of $G$; that is,
\begin{align*}
P(G)&=P(\deg(v_1),\deg(v_2),\dots,\deg(v_n))=\sum_{1\leq i<j\leq n}\min\{\deg(v_i),\deg(v_j)\}.
\end{align*}
\end{definition}

Recall that if $\kappa(u,v)=\min\{\deg(u),\deg(v)\}$ for all pairs of distinct vertices $u$ and $v$ of $G$, then we say that $G$ is ideally connected.  Since $\kappa(u,v)\leq \min\{\deg(u),\deg(v)\}$ for all $u,v$, we have $K(G)\leq P(G)$, with equality if and only if $G$ is ideally connected.

We first show that among all sequences of $n$ positive integers whose sum is a fixed number $D$, the sequence that maximizes the potential is nearly regular.

\begin{lemma}\label{Balanced}
Let $d_1,d_2,\dots,d_n$ be a sequence of positive integers, and let $D=\displaystyle\sum_{i=1}^n d_i$.  Let $D=dn+r,$ where $d\geq 0$ and $0\leq r<n.$  Then
\[
P(d_1,d_2,\dots,d_n)\leq P(\underbrace{d,\dots,d}_{\text{$n-r$ terms}},\underbrace{d+1,\dots,d+1}_{\text{$r$ terms}}).  \qedhere
\]
\end{lemma}

\begin{proof}
Assume, without loss of generality, that $d_1\leq d_2\leq \dots\leq d_n$, and suppose that $d_n-d_1\geq 2$.  It suffices to show that
\[
P(d_1,d_2,\dots,d_n)< P(d_1+1,d_2,\dots,d_{n-1},d_n-1).
\]
Note first that we have
\begin{align*}
P(d_1,d_2,\dots,d_n)&=\sum_{i=2}^n\min\{d_1,d_i\}+\sum_{i=2}^{n-1}\min\{d_i,d_n\}+P(d_2,\dots,d_{n-1})\\
&=(n-1)d_1+\sum_{i=2}^{n-1}d_i+P(d_2,\dots,d_{n-1}).
\end{align*}
Suppose that exactly the first $a$ terms $d_1,d_2,\dots,d_a$ are equal to $d_1$, and exactly the last $b$ terms $d_{n-b+1},\dots, d_n$ are equal to $d_n$.  Certainly, we have $a+b\leq n$ since $d_1<d_n$.  Then
\begin{align*}
&P(d_1+1,d_2,\dots,d_{n-1},d_n-1)\\
&=\sum_{i=2}^n\min\{d_1+1,d_i\}+\sum_{i=2}^{n-1}\min\{d_i,d_n-1\}+P(d_2,\dots,d_{n-1})\\
&=(a-1)d_1+(n-a)(d_1+1)+\sum_{i=2}^{n-b}d_i+(b-1)(d_n-1)+P(d_2,\dots,d_{n-1})\\
&=(n-1)d_1+(n-a)+\sum_{i=2}^{n-b}d_i+(b-1)d_n-(b-1)+P(d_2,\dots,d_{n-1})\\
&=(n-1)d_1+(n-a)+\sum_{i=2}^{n-1}d_i-(b-1)+P(d_2,\dots,d_{n-1})\\
&=(n-1)d_1+\sum_{i=2}^{n-1}d_i+P(d_2,\dots,d_{n-1})+(n-a)-(b-1)\\
&=P(d_1,d_2,\dots,d_n)+n-a-b+1\\
&\geq P(d_1,d_2,\dots,d_n)+1.\qedhere
\end{align*}
\end{proof}

We use the following result of Beineke, Oellermann, and Pippert \cite{BeinekeOellermannPippert2002} to establish sharpness of our upper bound.

\begin{theorem}[{\cite[Section 2]{BeinekeOellermannPippert2002}}] Let $n$ and $m$ be integers such that $3\le n\le m\le \binom{n}{2}$.  Then there is an ideally connected nearly regular simple graph of order $n$ and size $m$. \label{exists_ideally_connected}
\end{theorem}

In fact, we note that most ideally connected nearly regular (multi)graphs are simple. More precisely we make the following straightforward observation.

\begin{observation}\label{MostlySimple}
Let $G$ be a nearly regular ideally connected graph of order $n\geq 3$ and size $m\geq n.$ Then either $G$ is simple, or $G$ has exactly two vertices of maximum degree, this pair of vertices is joined by exactly two edges, and this is the only multiple edge.
\end{observation}

We are now ready to prove the main result of this section.

\begin{theorem}\label{KappaBarBound}
Let $G$ be a minimally $2$-connected graph of order $n$.  Then
\[
\overline{\kappa}(G)\leq 2+\tfrac{(n-2)^2}{4n(n-1)}<\tfrac{9}{4}.
\]
Moreover, let $n=4k+\ell$, where $k,\ell\in\mathbb{Z}$ and $0\leq \ell<4$.
\begin{enumerate}
\item If $k\geq 8$ and $\ell=0$, then
\[
\overline{\kappa}(G)\leq 2+\tfrac{n^2-4n}{4n(n-1)}=2+\tfrac{n-4}{4(n-1)},
\]
with equality if and only if $G$ is obtained from an ideally connected $6$-regular graph of order $k$ by subdividing every edge.
\item If $k\geq 30$ and $\ell=1$, then
\[
\overline{\kappa}(G)\leq 2+\tfrac{n^2-6n+13}{4n(n-1)},
\]
with equality if and only if $G$ is obtained from an ideally connected nearly regular (multi)graph of order $k$ and size $n-k=3k+1$ by subdividing every edge.
\item If $k\geq 68$ and $\ell=2$, then
\[
\overline{\kappa}(G)\leq 2+\tfrac{n^2-8n+60}{4n(n-1)},
\]
with equality if and only if $G$ is obtained from an ideally connected nearly regular graph of either order $k$ and size $n-k=3k+2$, or order $k+1$ and size $n-k-1=3k+1$, by subdividing every edge.
\item If $k\geq 30$ and $\ell=3$, then
\[
\overline{\kappa}(G)\leq 2+\tfrac{n^2-6n+17}{4n(n-1)},
\]
with equality if and only if $G$ is obtained from an ideally connected nearly regular graph of order $k+1$ and size $n-k-1=3k+2$ by subdividing every edge.
\end{enumerate}
\end{theorem}

\begin{proof}
Let $G$ be an optimal minimally $2$-connected graph of order $n$.  By Theorem~\ref{bipartite}, $G$ is a bipartite graph, with the set of vertices of degree $2$ and the set of vertices of degree exceeding $2$ being independent sets.  Let $H$ be the (multi)graph obtained from $G$ by replacing every vertex of degree $2$ with an edge between its neighbours, and note that $G$ can be recovered from $H$ by subdividing each of its edges.   Suppose that $G$ has $s$ vertices of degree at least $3$, and hence $n-s$ vertices of degree $2$.  Then $H$ has $s$ vertices and $n-s$ edges.  Note that $s\leq \tfrac{2}{5}n$, as the sum of the degrees of the $s$ vertices of degree at least $3$ must be equal to $2(n-s)$.
By a straightforward argument, we have
\begin{align*}
K(G)&=2\left[\tbinom{n}{2}-\tbinom{s}{2}\right]+K(H)\\
&\leq 2\left[\tbinom{n}{2}-\tbinom{s}{2}\right]+P(H),
\end{align*}
with equality if and only if $H$ is ideally connected.  Let $2(n-s)=ds+r$ for $d,r\in\mathbb{Z}$ and $0\leq r<s$.  Then, by Lemma~\ref{Balanced},
\begin{align*}
2\left[\tbinom{n}{2}-\tbinom{s}{2}\right]+P(H)&\leq 2\tbinom{n}{2}-2\tbinom{s}{2}+d\tbinom{s}{2}+\tbinom{r}{2}\\
&=2\tbinom{n}{2}+(d-2)\tbinom{s}{2}+\tbinom{r}{2}\\
&=n(n-1)+\left[\tfrac{2(n-s)-r}{s}-2\right]\tbinom{s}{2}+\tbinom{r}{2}\\
&=n(n-1)+\left[\tfrac{2n-4s-r}{s}\right]\tfrac{s(s-1)}{2}+\tfrac{r(r-1)}{2}\\
&=n(n-1)+(2n-4s)(s-1)/2-r(s-1)/2+r(r-1)/2\\
&=n(n-1)+(n-2s)(s-1)-r(s-r)/2,
\end{align*}
with equality if and only if $H$ is nearly regular (i.e., $H$ has $r$ vertices of degree $d+1$ and $s-r$ vertices of degree $d$).  So far, the bound on $K(G)$ is tight if and only if $H$ is ideally connected and nearly regular.  By Theorem \ref{exists_ideally_connected}, there exists such a graph $H$ (in fact, a simple graph) for any choice of $n$ and $s$ where $n-s \le \tbinom{s}{2}$.

To prove the general bound given in the theorem statement, we first observe, using elementary calculus, that $(n-2s)(s-1)$ achieves a maximum of $\tfrac{(n-2)^2}{8}$ at $s=\tfrac{n+2}{4}$. Thus
\[
K(G)\leq n(n-1)+(n-2s)(s-1)-r(s-r)/2\leq n(n-1)+(n-2s)(s-1)\leq n(n-1)+\tfrac{(n-2)^2}{8},
\]
Dividing through by $\binom{n}{2}$ gives the general upper bound on $\overline{\kappa}(G)$.

We now prove the exact upper bounds given by parts (a), (b), (c), and (d) of the theorem statement.  To do so, we determine the exact value(s) of $s$ at which the quantity
\[
(n-2s)(s-1)-r(s-r)/2
\]
is maximized, and we show that $n-s \le \tbinom{s}{2}$ at all such values, which guarantees that the maximum is actually attained by some graph.  We consider parts (a), (b), (c), and (d) separately.

For part (a), let $n=4k$ with $k\geq 8$.  We show that
\[
g_k(s)=(4k-2s)(s-1)-r(s-r)/2\leq 2k^2-2k=\tfrac{n^2-4n}{8},
\]
with equality if and only if $s=k$.  First, if $s=k$, then $d=6$ and $r=0$, and thus $g_k(k)=2k^2-2k.$  Next, if $s=k+1$, then $d=5$ and $r=k-7>0$, so $g_k(k+1)=2k^2-2k-4(k-7)<2k^2-2k.$  Lastly, if $s\not\in\{k,k+1\},$ let $f_k(s)=(4k-2s)(s-1)$.  Clearly $g_k(s)\leq f_k(s)$, and we show that $f_k(s)<2k^2-2k$ for all $s\not\in\{k,k+1\}$.  The function $f_k(s)$ is a quadratic in $s$ which attains its maximum value at $s=k+\tfrac{1}{2}$.  Thus, if $s<k$, then $f_k(s)<f_k(k)=2k^2-2k$, and if $s>k+1$, then $f_k(s)<f_k(k+1)=2k^2-2k$.

In conclusion, we have
\[
K(G)\leq n(n-1)+\tfrac{n^2-4n}{8},
\]
with equality if and only if $H$ is an ideally connected nearly regular (multi)graph on $k$ vertices and $n-k=3k$ edges (i.e., $H$ is $6$-regular).  By Observation~\ref{MostlySimple}, $H$ must be a simple graph.   Since $k \ge 8$, we have $n-k=3k \le \binom{k}{2}$, so indeed, Theorem~\ref{exists_ideally_connected} guarantees sharpness.  The bound on $\overline{\kappa}(G)$ follows by dividing through by $\binom{n}{2}$. This completes the proof of part (a).

For part (b), let $n=4k+1$ with $k\geq 30$.  We claim that
\[
g_k(s)=(4k+1-2s)(s-1)-r(s-r)/2\leq 2k^2-2k+1=\tfrac{n^2-6n+13}{8},
\]
with equality if and only if $s=k$.  First off, if $s=k$, then $d=6$ and $r=2$, and it follows that $g_k(k)=2k^2-2k+1$.  It remains to show that $g_k(s)<g_k(k)$ for all $s\neq k$.  We consider three cases.

\medskip
\noindent
\textbf{Case 1:} $s\in(k-\tfrac{k-2}{9},k)$

\noindent
Let $s=k-i$ for some integer $i\in[1,\tfrac{k-2}{9}).$  It follows that $d=6$ and $r=2+8i< k-i=s$ (note that $2+8i<k-i$ since $i<\tfrac{k-2}{9}$).  Now
\[
g_k(k-i)=34i^2+(14-4k)i+2k^2-2k+1
\]
is a quadratic in $i$ with positive leading coefficient, so for $i\in [1,\tfrac{k-2}{9})$,
\[
g_k(k-i)\leq \max\left\{g_k(k-1),g_k\left(k-\tfrac{k-2}{9}\right)\right\}.
\]
We verify that $g_k(k)>g_k(k-1)=2k^2-6k+49$ for all $k\geq 13$, and that $g_k(k)>g_k\left(k-\tfrac{k-2}{9}\right)=\tfrac{160}{81}k^2-\tfrac{100}{81}k-\tfrac{35}{81}$ for all $k\geq 30$.  Therefore, for $s\in(k-\tfrac{k-2}{9},k)$, we have $g_k(s)<g_k(k).$

\medskip
\noindent
\textbf{Case 2:} $s\in(k,k+\tfrac{k+2}{7}).$

\noindent
Let $s=k+i$ for some integer $i\in[1,\tfrac{k+2}{7}).$  It follows that $d=5$ and $r=k+2-7i$ (note that $k+2-7i<k+i$ since $i\geq 1$ and $k+2-7i>0$ since $i<\tfrac{k+2}{7}$).  Now
\[
g_k(k+i)=26i^2+(-12-4k)i+2k^2+1
\]
is a quadratic in $i$ with positive leading coefficient, so for $i\in [1,\tfrac{k+2}{7})$,
\[
g_k(k+i)\leq \max\left\{g_k(k+1),g_k\left(k+\tfrac{k+2}{7}\right)\right\}.
\]
We verify that $g_k(k)>g_k(k+1)=2k^2-4k+15$ for all $k\geq 8$, and $g_k(k)>g_k\left(k+\tfrac{k+2}{7}\right)=\tfrac{96}{49}k^2-\tfrac{36}{49}k+\tfrac{15}{49}$ for all $k\geq 30$.  Therefore, for $s\in(k,k+\tfrac{k+2}{7})$, we have $g_k(s)<g_k(k).$

\medskip
\noindent
\textbf{Case 3:} $s\leq k-\tfrac{k-2}{9}$ or $s\geq k+\tfrac{k+2}{7}$

\noindent
Let $f_k(s)=(4k+1-2s)(s-1)$, and we certainly have $g_k(s)\leq f_k(s)$ (with equality if and only $r(s-r)=0$).  By elementary calculus, $f_k(s)$ is increasing when $s<k$ and decreasing when $s>k+1$.  So if $s\leq k-\tfrac{k-2}{9}$, then
\[
g_k(s)\leq f_k(s)\leq f_k\left(k-\tfrac{k-2}{9}\right)=\tfrac{160}{81}k^2-\tfrac{100}{81}k-\tfrac{35}{81},
\]
which is strictly less than $g_k(k)$ for $k\geq 30$.  Similarly, if $s\geq k+\tfrac{k+2}{7}$, then
\[
g_k(s)\leq f_k(s)\leq f_k\left(k+\tfrac{k+2}{7}\right)=\tfrac{96}{49}k^2-\tfrac{36}{49}k-\tfrac{15}{49},
\]
which is strictly less than $g_k(k)$ for $k\geq 30.$

In conclusion, we have
\[
K(G)\leq n(n-1)+\tfrac{n^2-6n+13}{8},
\]
with equality if and only if $H$ is an ideally connected nearly regular (multi)graph on $k$ vertices and $n-k=3k+1$ edges.  One can verify that $H$ has exactly two vertices of maximum degree $7$, so by Observation~\ref{MostlySimple}, $H$ may have a single multiple edge between these vertices, but has no other multiple edges.  Since $k \ge 30$, we have $n-k=3k+1 \le \binom{k}{2}$, so indeed, Theorem~\ref{exists_ideally_connected} guarantees sharpness.  The bound on $\overline{\kappa}(G)$ follows by dividing through by $\binom{n}{2}$.  This completes the proof of part (b).

\medskip

The analogous statements for parts (c) and (d) are proven similarly.
\end{proof}

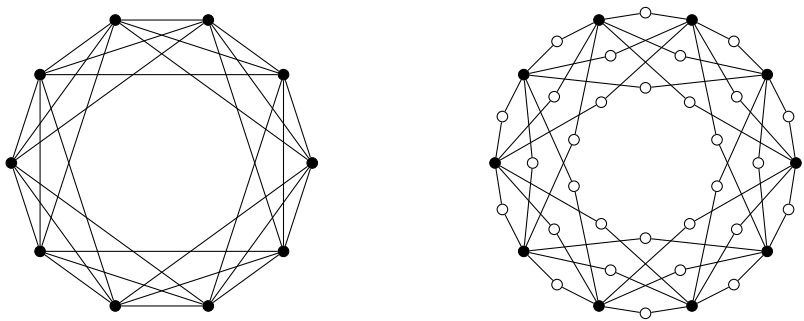
\begin{figure}[htb]
\begin{minipage}{\textwidth}
\centering{
\raisebox{-0.5\height}{\begin{tikzpicture}
\pgfmathsetmacro{\n}{10}
\pgfmathtruncatemacro{\m}{\n-1}
\foreach \x in {0,...,\m}
{
\vertex (\x) at (360*\x/\n:2cm) {};
}
\foreach \x in {0,...,\m}
{
\pgfmathtruncatemacro{\y}{Mod(\x+1,\n)}
\pgfmathtruncatemacro{\z}{Mod(\x+2,\n)}
\pgfmathtruncatemacro{\w}{Mod(\x+3,\n)}
\path
(\x) edge (\y)
(\x) edge (\z)
(\x) edge (\w);
}
\end{tikzpicture}}
\hspace{2cm}
\raisebox{-0.5\height}{\begin{tikzpicture}
\pgfmathsetmacro{\n}{10}
\pgfmathtruncatemacro{\m}{\n-1}
\foreach \x in {0,...,\m}
{
\vertex (\x) at (360*\x/\n:2cm) {};
\pgfmathtruncatemacro{\a}{\n+\x}
\pgfmathtruncatemacro{\b}{2*\n+\x}
\pgfmathtruncatemacro{\c}{3*\n+\x}
\hollowvertex (\a) at (180/\n+360*\x/\n:2cm) {};
\hollowvertex (\b) at (360/\n+360*\x/\n:1.5cm) {};
\hollowvertex (\c) at (540/\n+360*\x/\n:1cm) {};
}
\foreach \x in {0,...,\m}
{
\pgfmathtruncatemacro{\y}{Mod(\x+1,\n)}
\pgfmathtruncatemacro{\z}{Mod(\x+2,\n)}
\pgfmathtruncatemacro{\w}{Mod(\x+3,\n)}
\pgfmathtruncatemacro{\a}{\x+\n}
\pgfmathtruncatemacro{\b}{\x+2*\n}
\pgfmathtruncatemacro{\c}{\x+3*\n}
\path
(\x) edge (\a)
(\a) edge (\y)
(\x) edge (\b)
(\b) edge (\z)
(\x) edge (\c)
(\c) edge (\w);
}
\end{tikzpicture}}
}
\end{minipage}
\caption{The graph $C_{10}^{(3)}$ (left) and the graph $S_{40}$ (right) obtained by subdividing every edge of $C_{10}^{(3)}$.  The vertices resulting from subdivision are indicated by hollow circles.}\label{SubdividedCirculant}
\end{figure}

The ideally connected nearly regular graphs described in \cite{BeinekeOellermannPippert2002} can now be used to give explicit constructions of optimal minimally $2$-connected graphs of order $n$ in each of the parts of Theorem~\ref{KappaBarBound}. In part (a), where $n=4k$ with $k\geq 8$, the ideally connected nearly regular graph on $k$ vertices and $3k$ edges (i.e., ideally connected $6$-regular graph on $k$ vertices) described in~\cite{BeinekeOellermannPippert2002} is $C_k^{(3)}$ (the cube of the cycle $C_k$, obtained from $C_k$ by joining all pairs of vertices at distance at most $3$).  Let $S_n$ be the graph obtained by subdividing every edge of $C_k^{(3)}$.  Then $S_n$ is an optimal minimally $2$-connected graph of order $n$. See Figure~\ref{SubdividedCirculant} for a drawing of $S_n$ in the case where $n=40$.  The other cases can be described in a similar manner.

We make particular mention of the fact that in case (b), where $n=4k+1$ with $k\geq 30$, we can add any one edge to $C_{k}^{(3)}$ (even creating one multiple edge if we like) to produce an ideally connected nearly regular graph of order $k$ and size $3k+1$.  Subdividing every edge of such a graph gives an optimal minimally $2$-connected graph of order $n$.  So indeed, the ideally connected nearly regular graph in the statement of Theorem~\ref{KappaBarBound}(b) may be a multigraph.  In parts (a), (c), and (d), however, the ideally connected nearly regular graph will be simple.

Finally, if $G$ is a minimally $2$-connected graph of order $n$, where $n$ is a small value not covered by Theorem~\ref{KappaBarBound}, then with the notation used in the proof of Theorem~\ref{KappaBarBound}, the bound
\begin{align}\label{TotalConBound}
K(G)\leq n(n-1)+(n-2s)(s-1)-r(s-r)/2
\end{align}
still holds, with equality if and only if $H$ is an ideally conected nearly regular graph on $s$ vertices and $n-s$ edges.  The exact maximum value of the right-hand side of (\ref{TotalConBound}) can be determined by checking all possibilities for $s$.  From the work of~\cite{BeinekeOellermannPippert2002}, we can guarantee that this bound will be sharp as long as some value of $s$ at which the maximum occurs satisfies $n-s\leq \binom{n}{2}$.

\section{Average edge-connectivity of minimally 2-edge-connected graphs}\label{Edge}

In this section, we obtain results about the structure of edge-optimal minimally $2$-edge-connected graphs, and use this to prove a sharp upper bound on the average edge-connectivity of minimally $2$-edge-connected graphs.

We first recall some elementary properties of minimally 2-edge-connected graphs, given by Chaty and Chein in~\cite{ChatyChein1979}.  A non-trivial graph having no cut vertices is called \emph{nonseparable}, and the \emph{blocks} of a non-trivial graph $G$ are the maximal nonseparable subgraphs of $G$.

\begin{lemma}[\cite{ChatyChein1979}]\label{Properties}\
\begin{enumerate}
\item A connected graph $G$ is minimally $2$-edge-connected if and only if $G$ has no bridge and for each $e\in E(G)$, the graph $G-e$ has a bridge that separates the endvertices of $e$.
\item \label{Blocks} Every block of a minimally $2$-edge-connected graph is minimally $2$-edge-connected.
\item \label{ParallelCut} If $G$ is a minimally $2$-edge-connected graph, then $G$ has no triple edges, and if $G$ has a pair of parallel edges between vertices $u$ and $v$, then the removal of these two edges separates $u$ and $v$.
\item \label{GluingProperty} If $G$ and $H$ are two minimally $2$-edge-connected graphs, then the graph obtained from the disjoint union $G\cup H$ by identifying $u\in V(G)$ and $v\in V(H)$ is minimally $2$-edge-connected.
\end{enumerate}
\end{lemma}

A \emph{necklace} is a nonseparable minimally $2$-edge-connected simple graph.  A graph $G$ is \emph{extensible} between vertices $x$ and $y$ if the graph $G^z_{xy}$ obtained from $G$ by adding a new vertex $z$ and the edges $xz$ and $yz$ is minimally $2$-edge-connected.  The graph $G^z_{xy}$ is called an \emph{extension} of $G$ between $x$ and $y$ through $z$, and we refer to this operation as \emph{extending} $x$ and $y$ through $z$.

\begin{lemma}[{\cite[Corollary~2]{ChatyChein1979}}]\label{Extensible}
Let $G$ be a necklace.  For distinct vertices $x$ and $y$ in $G$, if $\lambda(x,y)\geq 3$, then $G$ is extensible between $x$ and $y$.
\end{lemma}

We also make use of the following straightforward lemma.

\begin{lemma}\label{Adjacent}
Let $G$ be a minimally $2$-edge-connected graph.  If $u$ and $v$ are adjacent in $G$, then $\lambda(u,v)=2$.
\end{lemma}

\begin{proof}
Let $e$ be an edge between $u$ and $v$.  Since $G$ is minimally $2$-edge-connected, $G-e$ has a bridge $e'$ that separates $u$ and $v$.  Every $u$-$v$ path in $G-e$ must contain $e'$, so there are at most two edge-disjoint paths between $u$ and $v$ in $G$.
\end{proof}

\subsection{Structural properties of edge-optimal minimally 2-edge-connected graphs}\label{Structure}
Recall that we call a minimally $2$-edge-connected graph of order $n$ having maximum average edge-connectivity an \emph{edge-optimal} minimally $2$-edge-connected graph of order $n$.  For $n\geq 5$, we prove that every edge-optimal minimally $2$-edge-connected graph $G$ is bipartite, with the set of vertices of degree $2$ and the set vertices of degree at least $3$ being the partite sets.  We also demonstrate that $G$ is $2$-connected, i.e., $G$ is a necklace.  First, we prove that the vertices of degree $2$ in an edge-optimal minimally $2$-edge-connected graph of order $n\geq 5$ form an independent set.  We use the following short lemma.

\begin{lemma}\label{ExtensiblePairExists}
Let $G$ be an edge-optimal minimally $2$-edge-connected graph of order $n\geq 5$.  Then $G$ contains a pair of vertices $x$ and $y$ that lie in the same block of $G$ and satisfy $\lambda(x,y)\geq 3$.
\end{lemma}

\begin{proof}
Since there is a minimally $2$-edge-connected graph on $n$ vertices with average edge-connectivity strictly greater than $2$ (take $K_{2,n-2},$ for example), and since $G$ has maximum average edge-connectivity among all such graphs, there is at least one pair of vertices $x,y$ in $G$ such that $\lambda(x,y)\geq 3$.  If $x$ and $y$ are in the same block, then we are done.  If $x$ and $y$ are not in the same block, then let $z$ be the first cut vertex that appears internally on every $x$--$y$ path.  Then $x$ and $z$ are in the same block and $\lambda(x,z)\geq 3$.
\end{proof}

\begin{theorem} \label{th:verticesdegree2}
Let $G$ be an edge-optimal minimally $2$-edge-connected graph of order $n\geq 5$.  Then no two vertices of degree $2$ are adjacent in $G$.
\end{theorem}

\begin{proof}
Suppose otherwise that $u$ and $v$ are adjacent vertices of degree $2$ in $G$.  By Lemma~\ref{ExtensiblePairExists}, there is a pair of vertices in $G$, say $x$ and $y$, such that $\lambda(x,y)\geq 3$ and $x$ and $y$ are in the same block.  Since $\deg(u)=\deg(v)=2,$ we know that $\{u,v\}\cap \{x,y\}= \emptyset.$  By Lemma~\ref{Extensible}, $G$ is extensible between $x$ and $y$.

Let $G_1$ be the graph obtained from $G$ by extending $x$ and $y$ through new vertex $z$, and let $G_2$ be the graph obtained from $G_1$ by contracting the edge $uv$ to vertex $w$.  Note that $G_2$ has order $n$.  We claim that
\begin{enumerate}[label=(\roman*)]
\item $G_2$ is minimally $2$-edge-connected, and
\item $\overline{\lambda}(G_2)>\overline{\lambda}(G)$,
\end{enumerate}
which contradicts the fact that $G$ is edge-optimal.

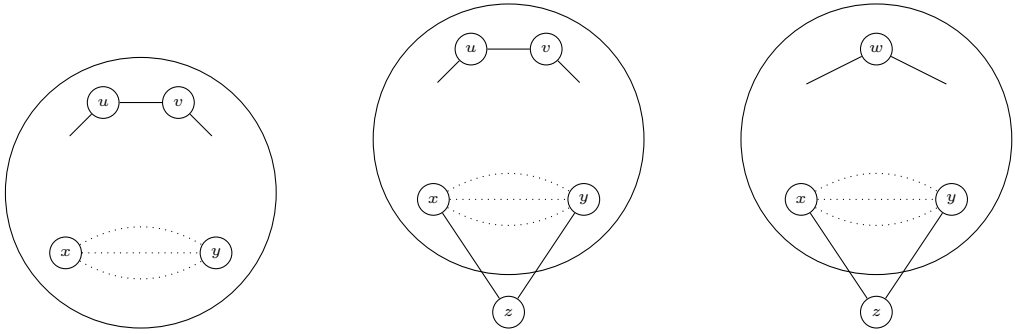
\begin{figure}[htb]
\centering{
\begin{tikzpicture}
\namedvertex (u) at (1,2) {\tiny $u$};
\namedvertex (v) at (2,2) {\tiny $v$};
\phantomvertex (un) at (0.5,1.5) {};
\phantomvertex (vn) at (2.5,1.5) {};
\namedvertex (x) at (0.5,0) {\tiny $x$};
\namedvertex (y) at (2.5,0) {\tiny $y$};
\path
(un) edge (u)
(vn) edge (v)
(u) edge (v)
;
\path[dotted]
(x) edge (y)
(x) edge[bend right] (y)
(x) edge[bend left] (y)
;
\draw (1.5,0.8) circle (1.8);
\end{tikzpicture}
\hspace{1cm}
\begin{tikzpicture}
\namedvertex (u) at (1,2) {\tiny $u$};
\namedvertex (v) at (2,2) {\tiny $v$};
\phantomvertex (un) at (0.5,1.5) {};
\phantomvertex (vn) at (2.5,1.5) {};
\namedvertex (x) at (0.5,0) {\tiny $x$};
\namedvertex (y) at (2.5,0) {\tiny $y$};
\namedvertex (vp) at (1.5,-1.5) {\tiny $z$};
\path
(un) edge (u)
(vn) edge (v)
(u) edge (v)
(x) edge (vp)
(vp) edge (y)
;
\path[dotted]
(x) edge (y)
(x) edge[bend right] (y)
(x) edge[bend left] (y)
;
\draw (1.5,0.8) circle (1.8);
\end{tikzpicture}
\hspace{1cm}
\begin{tikzpicture}
\namedvertex (u) at (1.5,2) {\tiny $w$};
\phantomvertex (un) at (0.5,1.5) {};
\phantomvertex (vn) at (2.5,1.5) {};
\namedvertex (x) at (0.5,0) {\tiny $x$};
\namedvertex (y) at (2.5,0) {\tiny $y$};
\namedvertex (vp) at (1.5,-1.5) {\tiny $z$};
\path
(un) edge (u)
(u) edge (vn)
(x) edge (vp)
(vp) edge (y)
;
\path[dotted]
(x) edge (y)
(x) edge[bend right] (y)
(x) edge[bend left] (y)
;
\draw (1.5,0.8) circle (1.8);
\end{tikzpicture}
}
\caption{The graphs $G,$ $G_1,$ and $G_2$ from left to right.  The dotted lines indicate edge-disjoint $x$-$y$ paths.  }\label{Degree2Picture}
\end{figure}

To see (i), first note that $G_1$ is minimally $2$-edge-connected, since $G$ is extensible between $x$ and $y$.  Now let $a$ and $b$ be distinct vertices in $V(G_2).$  If $w\not\in\{a,b\}$, then there are two edge-disjoint $a$-$b$ paths in $G_1$.  If either of these paths contains the edge $uv$, contract $uv$ to $w$ in this path, and we obtain two edge-disjoint $a$-$b$ paths in $G_2$.  If $w\in\{a,b\}$, then assume without loss of generality that $a=w$.  Then $b\in V(G_1)-\{u,v\}$, and there are two edge-disjoint $u$--$b$ paths in $G_1$, one of which must contain the edge $uv$.  These paths give rise to two edge-disjoint $w$--$b$ paths in $G_2$ when we contract $uv$ to $w$.  Therefore, $G_2$ is $2$-edge-connected.  Let $e$ be any edge in $G_2$.  If $e$ is incident to $w$, then since $\deg_{G_2}(w)=2$, the other edge incident to $w$ is a bridge in $G_2-e$.  Otherwise, if $e$ is not incident to $w$ in $G_2$, then $e$ must also be an edge in $G_1$.  Since $G_1$ is minimally $2$-edge-connected, $G_1-e$ has a bridge, say $f$.  If $f$ is not incident to $u$ or $v$, then $f$ is also a bridge in $G_2-e$.  Suppose otherwise that $f$ is incident to $u$ or $v$.  Then either edge incident to $w$ in $G_2$ is different from $e$ and is a bridge in $G_2-e$.

For (ii), first note that $u$ and $v$ each have degree $2$ in $G$, and $w$ and $z$ each have degree $2$ in $G_2$, so $\lambda_G(u,a)=\lambda_{G_2}(w,a)=2$ and $\lambda_G(v,a)=\lambda_{G_2}(z,a)=2$ for all $a\in V(G_2)-\{w,z\}$, and $\lambda_G(u,v)=\lambda_{G_2}(w,z)=2$.  Let $a,b\in V(G_2)-\{w,z\}=V(G)-\{u,v\}$.  Consider a collection $\mathcal{P}_{a,b}$ of $\lambda_G(a,b)$ edge-disjoint $a$-$b$ paths in $G.$  If any path in $\mathcal{P}_{a,b}$ contains the edge $uv$, contract $uv$ to $w$ in this path to obtain a collection of $\lambda_{G}(a,b)$ edge-disjoint $a$-$b$ paths in $G_2$.  So $\lambda_G(a,b)\leq \lambda_{G_2}(a,b)$.  Finally, $\lambda_G(x,y)< \lambda_{G_2}(x,y)$, since we have added the extra $x$-$y$ path through $z$.  We conclude that $\overline{\lambda}(G)<\overline{\lambda}(G_2)$.
\end{proof}

In order to prove that the vertices of degree at least $3$, in an edge-optimal minimally $2$-edge-connected graph, are independent, we require two lemmas concerning the structure within each block of an edge-optimal minimally $2$-edge-connected graph.  
The first lemma tells us that every block in an edge-optimal minimally $2$-edge-connected graph $G$ of order $n\geq 5$ has average edge-connectivity exceeding $2$.

\begin{lemma}\label{BlockContract}
Let $G$ be an edge-optimal minimally $2$-edge-connected graph of order $n\geq 5$.  Then for every block $B$ of $G$, $\overline{\lambda}(B)>2$; i.e., there is some pair of vertices in $B$ with edge-connectivity at least $3$.
\end{lemma}

\begin{proof}
Suppose otherwise that $G$ has a block $B$ with $\overline{\lambda}(B)=2$, and let $p=|V(B)|$ (note that $p\geq 2$ and that $B$ is a cycle). By Lemma~\ref{ExtensiblePairExists}, there is a pair of vertices in $G$, say $x$ and $y$, such that $\lambda(x,y)\geq 3$ and $x$ and $y$ are in the same block.  Evidently, $B$ does not contain both $x$ and $y$, so assume $y\not\in B$.  Let $G'$ be the graph obtained by extending $G$ a total of $p-1$ times between $x$ and $y$  through the new vertices $z_1,\dots, z_{p-1}$, and then contracting all vertices of $B$ to a single vertex $\beta$.  Note that $G'$ has order $n$, and is easily seen to be minimally $2$-edge-connected, by Lemma~\ref{Properties}\ref{Blocks}, Lemma~\ref{Properties}\ref{GluingProperty}, and Lemma~\ref{Extensible}.

We show that $\overline{\lambda}(G)<\overline{\lambda}(G')$, which contradicts the edge-optimality of $G$.  First of all, let $v\in V(G)-V(B)$.  Let $b_v$ be the unique vertex in $B$ at shortest distance from $v$ in $G$.  Then $\lambda_{G}(v,b)=2$ for all $b\in V(B)-b_v$, since $\lambda_G(b_v,b)=2$, by the assumption that $\overline{\lambda}(B)=2$.  Further, $\lambda_{G}(v,b_v)\leq \lambda_{G'}(v,\beta)$, since any collection of $\lambda_{G}(v,b_v)$ edge-disjoint $v$--$b_v$ paths in $G$ gives rise to $\lambda_G(v,b_v)$ edge-disjoint $v$--$\beta$ paths in $G'$ by replacing $b_v$ with $\beta$.  Therefore,
\[
\sum_{b\in V(B)}\lambda_G(v,b)=\lambda_G(v,b_v)+2(p-1)\leq \lambda_{G'}(v,\beta)+\sum_{i=1}^{p-1}\lambda_{G'}(v,z_i)
\]
for any $v\in V(G)-V(B)$.

Now let $u,v\in V(G)-V(B)$.  Then $\lambda_G(u,v)\leq \lambda_{G'}(u,v)$, since any collection of $\lambda_G(u,v)$ edge-disjoint $u$-$v$ paths in $G$ gives rise to a collection of $\lambda_G(u,v)$ edge-disjoint $u$-$v$ paths in $G'$ when we contract $B$ to $\beta$.  Therefore,
\[
\sum_{\{u,v\}\subseteq V(G)-V(B)}\lambda_G(u,v)\leq \sum_{\{u,v\}\subseteq V(G)-V(B)}\lambda_{G'}(u,v).
\]

If $a,b\in V(B)$, then $\lambda_G(a,b)=2$, and if $a',b'\in\{\beta,z_1,\dots,z_{p-1}\}$, then $\lambda_{G'}(a',b')=2$ since $\deg_{G'}(z_i)=2$ for all $i\in\{1,\dots,p-1\}$. Therefore, the total connectivity in $G$ between all vertices in $V(B)$ is equal to the total connectivity in $G'$ between all vertices of $\{\beta,z_1,\dots,z_{p-1}\}$.

Finally, we have $\lambda_G(x,y)< \lambda_{G'}(x,y)$ (or $\lambda_G(x,y)< \lambda_{G'}(\beta,y)$ if $x\in B$).  This is due to the $p-1>0$ new $x$--$y$ paths through the vertices $z_i$, which are not counted above.
\end{proof}

The following corollary is nearly immediate.

\begin{corollary}\label{Simple}
Let $G$ be an edge-optimal minimally $2$-edge-connected graph of order $n\geq 5$.  Then $G$ is simple.
\end{corollary}

\begin{proof}
Suppose, towards a contradiction, that $G$ has a pair of parallel edges $e_1$ and $e_2$ between vertices $u$ and $v$.  Then by Lemma~\ref{Properties}\ref{ParallelCut}, the vertices $u$ and $v$ make up a block of $G$ with average edge-connectivity $2$.  This contradicts Lemma~\ref{BlockContract}.
\end{proof}

So in the remainder of this section, we describe paths in edge-optimal minimally $2$-edge-connected graphs of order at least $5$ by listing only the vertices.  The next lemma describes a property of every cut vertex of an edge-optimal minimally $2$-edge-connected graph.

\begin{lemma}\label{CutConnectivity}
Let $G$ be an edge-optimal minimally $2$-edge-connected graph of order $n\geq 5$.  If $G$ has a cut vertex $v$, then every block of $G$ containing $v$, has some vertex $w\neq v$ such that $\lambda_G(v,w)\geq 3$.
\end{lemma}

\begin{proof}
Let $v$ be a cut vertex of $G$, and let $H_1,\dots, H_p$ be the components of $G-v$. For every $i\in\{1,\dots,p\},$ let $H'_i$ be the subgraph of $G$ induced by $V(H_i)\cup\{v\}$.    By Lemma~\ref{Properties}, $H'_i$ is a minimally $2$-edge-connected graph.    Note also that there are exactly $p$ blocks of $G$ containing $v$; let $B_i$ be the block of $G$ containing $v$ that is a subgraph of $H'_i$.  Suppose, towards a contradiction, that $\lambda_G(v,w)=2$ for all $w\in V(B_i)$ for some $i$, $1 \le i \le p$. Without loss of generality $i=1$, i.e., $w\in V(B_1)$.

We now describe a construction of a graph $G'$ that is minimally $2$-connected with average connectivity exceeding that of $G$. Relabel the copy of $v$ in $H'_i$ with the label $v_i$. For every $i\in\{1,\dots,p\}$, if there is a vertex $w_i\in B_i$ such that $\lambda_{G}(v,w_i)\geq 3$, then define $u_i=v_i$.  Otherwise, by Lemma~\ref{BlockContract}, there is some pair of vertices in $B_i-v_i$, say $x_i$ and $y_i$, such that $\lambda_G(x_i,y_i)\geq 3$. In this case, define $u_i=x_i$ (whether $x_i$ or $y_i$ is chosen does not matter).  Since $\lambda_G(v,w)=2$ for all $w\in B_1$, we see that $u_1=x_1\neq v_1$.  Let $G'$ be the graph obtained from the disjoint union $\bigcup_{i=1}^k H'_i$ by identifying all vertices in the set $\{u_1,\dots,u_k\}$.  By Lemma~\ref{Properties}\ref{GluingProperty}, $G'$ is a minimally $2$-edge-connected graph of order $n$, and it is straightforward to verify that $\lamb(G')>\lamb(G)$, which contradicts the fact that $G$ is edge-optimal.
\end{proof}

We are now ready to prove that vertices of degree at least $3$ are independent in every edge-optimal minimally $2$-edge-connected graph of order $n\geq 5$.

\begin{theorem} \label{th:verticesdegreegreater2}
Let $G$ be an edge-optimal minimally $2$-edge-connected graph of order $n\geq 5$.    Then no two vertices of degree at least $3$ are adjacent in $G$.
\end{theorem}

\begin{proof}
Suppose otherwise that $u$ and $v$ are adjacent vertices of degree at least $3$ in $G$.  Since $G$ is minimally $2$-edge-connected, the graph $G-uv$ has a bridge, say $xy$.  So $G-\{uv,xy\}$ has exactly two connected components, say $G_1$ containing $u$ and $x$, and $G_2$ containing $v$ and $y$.  Let $G'$ be the graph obtained from $G$ by contracting $u$ and $v$ to a single vertex $w$, and subdividing the edge $xy$. Let $z$ denote the new vertex between $x$ and $y$.

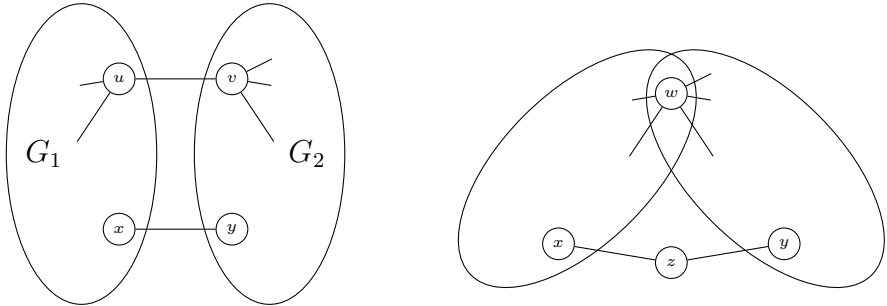
\begin{figure}[htb]
\centering{
\begin{tikzpicture}
\namedvertex (u) at (0,2.5) {\tiny $u$};
\namedvertex (v) at (1.5,2.5) {\tiny $v$};
\namedvertex (w) at (0,0.5) {\tiny $x$};
\namedvertex (z) at (1.5,0.5) {\tiny $y$};
\phantomvertex (un1) at (-0.6,2.4) {};
\phantomvertex (un2) at (-0.6,1.6) {};
\phantomvertex (vn1) at (2.1, 2.4) {};
\phantomvertex (vn2) at (2.1,1.6) {};
\phantomvertex (vn3) at (2.1, 2.8) {};
\path
(u) edge (un1)
(u) edge (un2)
(v) edge (vn1)
(v) edge (vn2)
(v) edge (vn3)
(u) edge (v)
(w) edge (z)
;
\draw (-0.5,1.5) ellipse (1 and 2);
\draw (2,1.5) ellipse (1 and 2);
\node at (-1,1.5) {$G_1$};
\node at (2.5,1.5) {$G_2$};
\end{tikzpicture}
\hspace{1cm}
\begin{tikzpicture}
\namedvertex (u) at (0.75,2.5) {\tiny $w$};
\namedvertex (w) at (-0.75,0.5) {\tiny $x$};
\namedvertex (z) at (2.25,0.5) {\tiny $y$};
\namedvertex (vp) at (0.75,0.25) {\tiny $z$};
\phantomvertex (un1) at (0.15,2.4) {};
\phantomvertex (un2) at (0.15,1.6) {};
\phantomvertex (vn1) at (1.35, 2.4) {};
\phantomvertex (vn2) at (1.35,1.6) {};
\phantomvertex (vn3) at (1.35, 2.8) {};
\path
(u) edge (un1)
(u) edge (un2)
(u) edge (vn1)
(u) edge (vn2)
(u) edge (vn3)
(w) edge (vp)
(vp) edge (z)
;
\draw[rotate around={-45:(-0.5,1.5)}] (-0.5,1.5) ellipse (1 and 2);
\draw[rotate around={45:(2,1.5)}] (2,1.5) ellipse (1 and 2);
\end{tikzpicture}
}
\caption{The graphs $G$ (left) and $G'$ (right).}\label{Degree3Picture}
\end{figure}

Note that $G$ and $G'$ have the same order.  We claim that $G'$ is minimally $2$-edge-connected, and that $\lamb(G')>\lamb(G)$, which contradicts the assumption that $G$ is edge-optimal.  First we show that $\lamb(G')>\lamb(G)$.  We demonstrate the following:
\begin{enumerate}[label=(\roman*)]
\item $\lambda_G(u,v)=\lambda_{G'}(w,z)$;
\item $\lambda_G(a,u)+\lambda_G(a,v)\leq \lambda_{G'}(a,w)+\lambda_{G'}(a,z)$ for all $a\in V(G)-\{u,v\}$;
\item $\lambda_G(a,b)\leq \lambda_{G'}(a,b)$ for all $a,b\in V(G)-\{u,v\}$; and
\item there exist vertices $a\in V(G_1)-u$ and $b\in V(G_2)-v$ such that $\lambda_{G}(a,b)< \lambda_{G'}(a,b)$.
\end{enumerate}
Summing the left-hand sides of (i), (ii), and (iii) over all possibilities gives the total connectivity of $G$, while summing the right-hand sides gives the total connectivity of $G'$, so together, (i)--(iv) give $\lamb(G)<\lamb(G')$.

By Lemma~\ref{Adjacent}, $\lambda_G(u,v)=2$.  Since $G_1$ and $G_2$ are connected, there is a $u$--$x$ path in $G_1$ and a $v$--$y$ path in $G_2$.  These paths give rise to two internally disjoint $w$--$z$ paths in $G'$ in the obvious manner, so $\lambda_{G'}(w,z)=2$ as well.  This completes the proof of (i).

For (ii), let $a\in V(G)-\{u,v\}$, and suppose without loss of generality that $a\in G_1$.  Then $\lambda_{G}(a,v)=2$ since the edges $uv$ and $xy$ separate $a$ and $v$.  Moreover, let $C$ be a cycle of $G$ formed from two edge disjoint $a$--$v$ paths.  Then $C$ gives rise to a cycle of $G'$ containing $a$ and $z$, so $\lambda_G'(a,z)\leq 2$ (in fact, $\lambda_{G'}(a,z)=2$ since $\deg_{G'}(z)=2$).  So $\lambda_G(a,v)=\lambda_{G'}(a,z)$.  Now let $\mathcal{P}_{a,u}$ be a collection of $\lambda_G(a,u)$ edge-disjoint $a$--$u$ paths in $G$.  At most one member of $\mathcal{P}_{a,b}$ contains the edge $uv$ (in which case it must also contain $xy$).  If such a path exists in $\mathcal{P}_{a,b}$, then performing the contraction of $uv$ to $w$ and the subdivision of $xy$ on this path and leaving all other paths in $\mathcal{P}_{a,b}$ as is, gives a collection of $\lambda_{G}(a,u)$ edge-disjoint $a$--$w$ paths in $G'$, so  $\lambda_{G}(a,u)\leq \lambda_{G'}(a,w)$ (in fact, equality is easily verified).

For (iii), let $a,b\in V(G)-\{u,v\}=V(G')-\{w,z\}$.  Let $\mathcal{P}_{a,b}$ be a collection of  $\lambda_G(a,b)$ edge-disjoint $a$--$b$ paths in $G$.  The edge $uv$ appears in at most one path in $\mathcal{P}_{a,b}$, and the edge $xy$ appears in at most one path in $\mathcal{P}_{a,b}$.  One obtains a collection of $\lambda_G(a,b)$ edge-disjoint $a$--$b$ paths in $G'$ from $\mathcal{P}_{a,b}$ by contracting such an appearance of $uv$ to $w$ and subdividing such an appearance of $xy$.

Finally, we prove (iv).  We find $a\in V(G_1)-u$ such that $\lambda_G(a,u)\geq 3$, and $b\in V(G_2)-v$ such that $\lambda_G(v,b)\geq 3$.  We then show that $\lambda_{G'}(a,b)\geq 3$.  Since  $a\in V(G_1)$ and $b\in V(G_2)$, we will have $\lambda_{G}(a,b)=2$, giving (iv).

Since $\{uv,wz\}$ is a cutset of $G$, it must be the case that $uv$ and $wz$ are contained in the same block $B$ of $G$.  We claim that $\deg_B(u)\geq 3$ and $\deg_B(v)\geq 3$.  If $u$ is not a cut vertex of $G$, then $\deg_B(u)=\deg_G(u)\geq 3$.  Otherwise, if $u$ is a cut vertex of $G$, then $\deg_B(u)\geq 3$, by Lemma~\ref{CutConnectivity}.  The proof is the same for $v$.  Now Since $B$ is $2$-connected, there is a cycle $C$ in $B$ containing both $uv$ and $wz$.  Let $u_1$ be the neighbour of $u$ on $C$ other than $v$, and let $u_2\neq u_1,v$ be another neighbour of $u$ in $B$.  In an analogous manner we find  corresponding neighbours $v_1$ and $v_2$ of $v$.  Let $B_i$ be the subgraph of $B$ induced by $V(G_i)\cup V(B)$ for $i=1,2$.
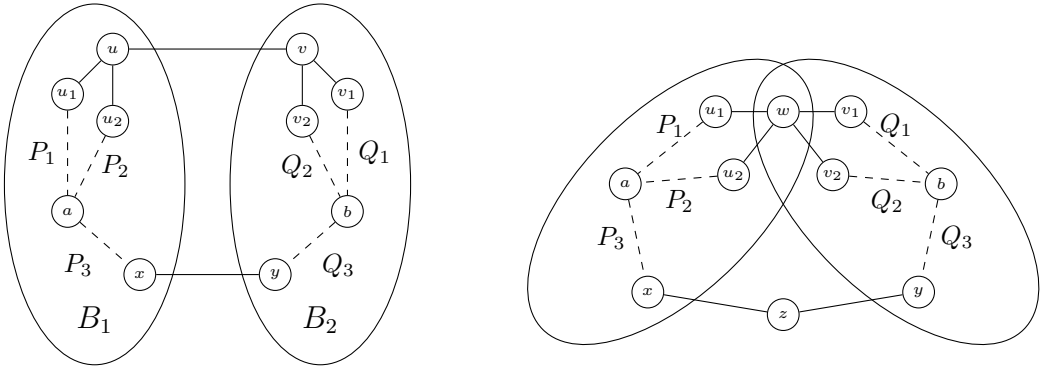
\begin{figure}[htb]
\centering{
\begin{tikzpicture}[scale=1.2]
\namedvertex (u) at (-0.3,3) {\tiny $u$};
\namedvertex (v) at (1.8,3) {\tiny $v$};
\namedvertex (w) at (0,0.5) {\tiny $x$};
\namedvertex (z) at (1.5,0.5) {\tiny $y$};
\namedvertex (un1) at (-0.8,2.5) {\tiny $u_1$};
\namedvertex (un2) at (-0.3,2.2) {\tiny $u_2$};
\namedvertex (vn1) at (2.3, 2.5) {\tiny $v_1$};
\namedvertex (vn2) at (1.8,2.2) {\tiny $v_2$};
\namedvertex (a) at (-0.8,1.2) {\tiny $a$};
\namedvertex (b) at (2.3,1.2) {\tiny $b$};
\path
(u) edge (v)
(w) edge (z)
(u) edge (un1)
(u) edge (un2)
(v) edge (vn1)
(v) edge (vn2)
;
\path[dashed]
(un1) edge [left] node {\footnotesize $P_1$} (a)
(un2) edge [right] node {\footnotesize $P_2$}(a)
(a) edge [below left] node {\footnotesize $P_3$}(w)
(vn1) edge [right] node {\footnotesize $Q_1$} (b)
(vn2) edge [left] node {\footnotesize $Q_2$} (b)
(b) edge [below right] node {\footnotesize $Q_3$} (z)
;
\draw (-0.5,1.5) ellipse (1 and 2);
\draw (2,1.5) ellipse (1 and 2);
\node at (-0.5,0) {$B_1$};
\node at (2,0) {$B_2$};
\def\CircArrowRight#1{\tikz[baseline=(A.base)]
  \draw[-stealth]
    (0,0) node[circle, inner sep=0cm](A){$#1$}
    let \p1=(A.center),\p2=(A.west), \n1={2*(\x1-\x2)} in
      (-90:\n1) arc (-90:190:\n1);}
\end{tikzpicture}
\hspace{1cm}
\begin{tikzpicture}[scale=1.2]
\namedvertex (u) at (0.75,2.5) {\tiny $w$};
\namedvertex (w) at (-0.75,0.5) {\tiny $x$};
\namedvertex (z) at (2.25,0.5) {\tiny $y$};
\namedvertex (vp) at (0.75,0.25) {\tiny $z$};
\namedvertex (un1) at (0,2.5) {\tiny $u_1$};
\namedvertex (un2) at (0.2,1.8) {\tiny $u_2$};
\namedvertex (vn1) at (1.5, 2.5) {\tiny $v_1$};
\namedvertex (vn2) at (1.3,1.8) {\tiny $v_2$};
\namedvertex (a) at (-1,1.7) {\tiny $a$};
\namedvertex (b) at (2.5,1.7) {\tiny $b$};
\path
(u) edge (un1)
(u) edge (un2)
(u) edge (vn1)
(u) edge (vn2)
(w) edge (vp)
(vp) edge (z)
;
\path[dashed]
(un1) edge [above] node {\footnotesize $P_1$} (a)
(un2) edge [below] node {\footnotesize $P_2$}(a)
(a) edge [left] node {\footnotesize $P_3$}(w)
(vn1) edge [above] node {\footnotesize $Q_1$} (b)
(vn2) edge [below] node {\footnotesize $Q_2$} (b)
(b) edge [right] node {\footnotesize $Q_3$} (z)
;
\draw[rotate around={-45:(-0.5,1.5)}] (-0.5,1.5) ellipse (1 and 2);
\draw[rotate around={45:(2,1.5)}] (2,1.5) ellipse (1 and 2);
\end{tikzpicture}
}
\caption{The structure of the block $B$ of $G$ (left) that gives rise to three edge-disjoint $a$--$b$ paths in $G'$ (right).  The dashed lines indicate edge-disjoint paths.}\label{ThreePaths}
\end{figure}

Let $a$ be a vertex on the cycle $C$ at minimum distance from $u_2$ in $B-u$ (note that $a\in V(G_1)$), and let $P_2$ be a shortest $a$--$u_2$ path in $G_1$.  Let $P_1$ be the $a$--$u_1$ path contained in $B_1$ described by $C$, and let $Q_1$ be the $v_1$--$b$ path contained in $B_2$ described by $C$.  Since $B$ is $2$-connected, the graph $B-u$ is connected.  Note that $P_2$ is edge-disjoint from $C$.  Similarly, one finds a vertex $b\in V(B_2)-v$ and a $v_2$--$b$ path $Q_2$ in $B_2$ that is edge-disjoint from $C$.  Finally, let $P_3$ be the $a$--$x$ path in $B_1$ described by $C$, and let $Q_3$ be the $y$--$b$ path in $B_2$ described by $C$.  We conclude that $P_1wQ_1$, $P_2wQ_2$, and $P_3zQ_3$ are edge-disjoint $a$--$b$ paths in $G'$ (see Figure~\ref{ThreePaths}).


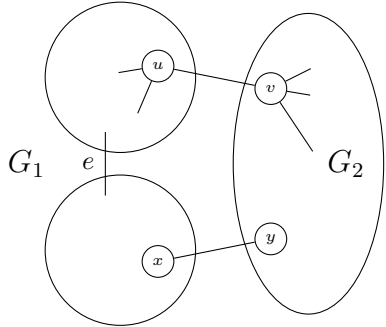
\begin{figure}[htb]
\centering{
\begin{tikzpicture}
\namedvertex (u) at (0,2.8) {\tiny $u$};
\namedvertex (v) at (1.5,2.5) {\tiny $v$};
\namedvertex (w) at (0,0.2) {\tiny $x$};
\namedvertex (z) at (1.5,0.5) {\tiny $y$};
\phantomvertex (un1) at (-0.6,2.7) {};
\phantomvertex (un2) at (-0.3,2.1) {};
\phantomvertex (vn1) at (2.1, 2.4) {};
\phantomvertex (vn2) at (2.1,1.6) {};
\phantomvertex (vn3) at (2.1, 2.8) {};
\phantomvertex (e1) at (-0.7,2) {};
\phantomvertex (e2) at (-0.7, 1) {};
\path
(u) edge (un1)
(u) edge (un2)
(v) edge (vn1)
(v) edge (vn2)
(v) edge (vn3)
(u) edge (v)
(w) edge (z)
(e1) edge node[left] {\footnotesize $e$} (e2)
;
\draw (-0.5,2.65) ellipse (1 and 1);
\draw (-0.5,0.35) ellipse (1 and 1);
\draw (2,1.5) ellipse (1 and 2);
\node at (-1.75,1.5) {$G_1$};
\node at (2.5,1.5) {$G_2$};
\end{tikzpicture}
}
\caption{The graph $G$ in the case that $uv$ is a bridge in $G-e$ (and so is $xy$).}\label{bequalsuv}
\end{figure}

It remains to show that $G'$ is minimally $2$-edge-connected.  When proving that $\lamb(G)<\lamb(G')$, we also established the fact that $G'$ is $2$-edge-connected.  Let $e$ be any edge in $G'$. We show that $G'-e$ has a bridge.  If $e=xz$, then $yz$ is a bridge in $G'-e$, and vice versa.  So we may assume that $e\in E(G')-\{xz,yz\}.$ So either $e \in E(G)-\{uv,xy\}$, or $e$ is incident with $w$, i.e., $e=wa$ where $ua$ or $va$ is in $E(G) -\{uv,xy\}$.
Suppose, without loss of generality, that $e\in E(G_1)$ or $e=wa$ where $ua \in E(G_1)$.  If $e \in E(G_1-u)$, then $G-e$ has a bridge $b$, since $G$ is minimally $2$-edge-connected.  If $b\not\in \{uv,xy\}$, then $b$ is also a bridge in $G'-e$.  If $b=xy$, then $xz$ is a bridge in $G'-e$.  Finally, if $b=uv$, then $xy$ is also a bridge in $G-e$ (see Figure~\ref{bequalsuv}), and hence $xz$ is a bridge in $G'-e$ once again.  If $e=wa$ and $a \in V(G_1)$, then $G-ua$ has a bridge $b$. We can argue as in the previous case that $G'-e$ has a bridge. Therefore, $G'$ is minimally $2$-edge-connected.
\end{proof}

As an immediate consequence of Theorem \ref{th:verticesdegree2} and Theorem \ref{th:verticesdegreegreater2} we have the following structure result for edge-optimal minimally $2$-edge-connected graphs.

\begin{corollary}\label{edge-optimal_are-bipartite}
Let $G$ be an edge-optimal minimally $2$-edge-connected graph of order $n\geq 5$.  Then $G$ is bipartite with partite sets the set of vertices of degree $2$ and the set of vertices of degree exceeding $2$.
\end{corollary}

We close this section with a proof that every edge-optimal minimally $2$-edge-connected graph of order $n\geq 5$ is $2$-connected.  The following observation will be useful in the proof.

\begin{observation}\label{permuting_paths_through_cut-vertex}
Let $G$ be a graph with cut vertex $x$ and let $u$ and $v$ be neighbours of $x$ from distinct components of $G-x$. Let $k=\lambda_G(u,v)$ and $P_1, P_2, \ldots P_k$ a family of $k$ edge-disjoint $u$--$v$ paths in $G$. Then $x$ necessarily lies on each $P_i$. Let $P_i'$ be the $u$--$x$ subpath of $P_i$ and $P_i''$ the $x$--$v$ subpath of $P_i$. If $\pi$ is any permutation of $1,2, \ldots, k$ then the collection of paths $\{P_i'\odot P_{\pi(i)}''| 1\le i \le k\}$ is also a set of $k$ edge-disjoint $u$--$v$ paths.
\end{observation}

\begin{theorem} \label{th:G2connected}
Let $G$ be an edge-optimal minimally $2$-edge-connected graph of order $n\geq 5$.  Then $G$ is $2$-connected.
\end{theorem}

\begin{proof}
Assume, to the contrary, that $G$ is not $2$-connected.  Let $x$ be a cut vertex of $G$, let $C_1$ be a component of $G-x$, and let $C_2$ be the union of the remaining components of $G-x$. Let $G_i=G[V(C_i) \cup \{ x \}]$, for $i=1, 2$. Observe that  $G_1$ and $G_2$ are both minimally $2$-edge-connected, since $G$ is minimally $2$-edge-connected.  By Lemma~\ref{CutConnectivity}, we may assume that  $x$ has degree at least $3$ in every block that contains it. Thus $\deg_{G_i}(x) \ge 3$ for $i=1,2$.

Let $u$ be a neighbour of $x$ in $G_1$ and let $v$ be a neighbour of $x$ in $G_2$. By Theorem~\ref{th:verticesdegreegreater2}, $\deg_{G}(u) = 2=\deg_{G}(v)$. So $\deg_{G_1}(u) = 2$ and $\deg_{G_2}(v) = 2$.  Let $y=N_{G_1}(u)-\{ x \}$ and $z=N_{G_2}(v)-\{ x \}$. Let $G'$ be obtained from $G$ by removing the edges $ux$ and $vx$  and adding the two new edges $uz$ and $vy$ (see Figure~\ref{constructing_G'}). Note that $G'$ is a simple graph of the same order as $G$.
\begin{figure}[htb]
\begin{center}
\begin{tikzpicture}
\draw (-1.5,0) ellipse (1.5 and 2);
\draw (1.5,0) ellipse (1.5 and 2);
\filldraw[fill=white] (0,0) circle (6pt);
\namedvertex (u) at (-2,0) {\tiny $y$};
\namedvertex (u1) at (-1,1) {\tiny $u$};
\namedvertex (x) at (0,0) {\tiny $x$};
\namedvertex (v1) at (1,1) {\tiny $v$};
\namedvertex (v) at (2,0) {\tiny $z$};
\path
(u) edge (u1)
(u1) edge (x)
(x) edge (v1)
(v1) edge (v)
;
\path[dotted]
(u) edge node[below]{\footnotesize $Q_1$} (x)
(v) edge node[below]{\footnotesize $Q_2$} (x)
;
\node at (-1.5,-1.5) {$G_1$};
\node at (1.5,-1.5) {$G_2$};
\end{tikzpicture}
\hspace{1cm}
\begin{tikzpicture}
\draw (-1.5,0) ellipse (1.5 and 2);
\draw (1.5,0) ellipse (1.5 and 2);
\filldraw[fill=white] (0,0) circle (6pt);
\namedvertex (u) at (-2,0) {\tiny $y$};
\namedvertex (u1) at (-1,1) {\tiny $u$};
\namedvertex (x) at (0,0) {\tiny $x$};
\namedvertex (v1) at (1,1) {\tiny $v$};
\namedvertex (v) at (2,0) {\tiny $z$};
\path
(u) edge (u1)
(u1) edge (v)
(u) edge (v1)
(v1) edge (v)
;
\path[dotted]
(u) edge node[below]{\footnotesize $Q_1$} (x)
(v) edge node[below]{\footnotesize $Q_2$} (x)
;
\end{tikzpicture}
\caption{The graphs $G$ (left) and $G'$ (right).  The dotted lines indicate paths.}
\label{constructing_G'}
\end{center}
\end{figure}
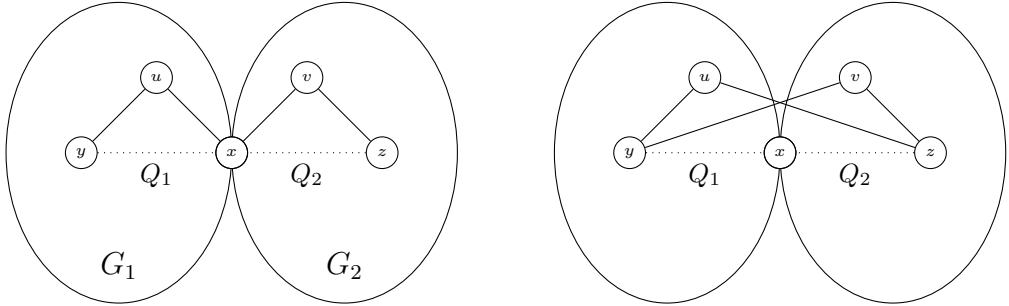

We will show that $\lamb(G)<\lamb(G')$, and that $G'$ is minimally $2$-edge-connected, which contradicts the assumption that $G$ is edge-optimal.  We first show that $\lamb(G)<\lamb(G')$ by demonstrating that $\lambda_G(a,b)\leq \lambda_{G'}(a,b)$ for all $a,b\in V(G)$, and that the inequality is strict for at least one pair.  We break the argument into the following cases:

\begin{enumerate}[label=(\roman*)]
\item $\lambda_{G}(y,z)<\lambda_{G'}(y,z)$,
\item $\lambda_{G}(u,v)=\lambda_{G'}(u,v)$,
\item $\lambda_{G}(a,u)=\lambda_{G'}(a,u)$ for all $a\in V(G)-\{u,v\}$,
\item $\lambda_{G}(a,v)=\lambda_{G'}(a,v)$ for all $a\in V(G)-\{u,v\}$,
\item $\lambda_{G}(a,b)\leq \lambda_{G'}(a,b)$ for all $a, b \in V(G)-\{u,v\}$.
\end{enumerate}

Throughout, we use the observation that there is an $x$--$y$ path $Q_1$ contained in $G_1-u$ (see Figure~\ref{constructing_G'}).  To see this, note that $G_1$ is $2$-edge-connected, so the removal of edge $xu$ does not separate $x$ and $y$.  So there is an $x$--$y$ path in $G_1$ that does not contain the edge $xu$, and in fact,  since $\deg_G(u)=2$, this path cannot contain $u$.  Similarly, there is an $x$--$z$ path $Q_2$ contained in $G_2-v$.

For (i), since $G$ is $2$-edge-connected, there exist at least two edge-disjoint $y$--$z$ paths in $G$. By Observation \ref{permuting_paths_through_cut-vertex}, we may assume that $P:yuxvz$ is a $y$--$z$ path in some family $\mathcal{F}$ of $\lambda_G(y,z)$ edge-disjoint $y$--$z$ paths in $G$. Let $\mathcal{F}'=(\mathcal{F}-\{P\}) \cup \{yuz, yvz\}$. Then $\mathcal{F}'$ is a family of $\lambda_G(y,z)+1$ edge-disjoint $y$--$z$ paths in $G'$.  So (i) holds.

For (ii), since $\deg_{G}(u)=\deg_G(v)=2$, we have $\lambda_G(u,v)=2$.  In $G'$, we have edge-disjoint $u$--$v$ paths $uyv$ and $uzv$, so $\lambda_{G'}(u,v)=2$, and we are done.

For (iii), let $a \in V(G)-\{u,v\}$.  Suppose first that $a \in V(G_1)-\{u\}$.  Since $G_1$ is (minimally) $2$-edge-connected, there exist two edge-disjoint $a$--$u$ paths $P_1$ and $P_2$ in $G_1$. We may assume that $y$ is the penultimate vertex of $P_1$ and $x$ is the penultimate vertex of $P_2$. Let $P_1'=P_1$ and let $P_2'=P_2[a,x]\odot Q_2\odot zu$. Then $P_1'$ and $P_2'$ are edge-disjoint $a$--$u$ paths in $G'$. So $\lambda_{G'}(a,u) =2$.  Suppose now that $a\in V(G_2)-\{v\}$.  There exist two edge-disjoint $a$--$v$ paths in $G_2$, and we may assume that $x$ is the penultimate vertex of $P_1$ and $z$ is the penultimate vertex of $P_2$.  Let $P'_1=P_1[a,x]\odot Q_1\odot yu$ and $P'_2=P_2[a,z]\odot zu$.  Then $P'_1$ and $P'_2$ are edge-disjoint $a$--$u$ paths in $G'$, so $\lambda_{G'}(a,u)=2$.  This completes the proof of (iii), and (iv) follows by symmetry.

For (v), let $a,b \in V(G)-\{u,v\}$.  Suppose first that we have $a,b\in V(G_1)-\{u\}$ (the case $a,b\in V(G_2)-\{v\}$ is similar).  Let $\mathcal{F}$ be a family of $\lambda_G(a,b)$ edge-disjoint paths in $G$. Since $a$ and $b$ are both in $G_1$ and $x$ is a cut vertex, these paths are contained in $G_1$. If none of these paths contain $u$, then $\mathcal{F}$ is a family of edge-disjoint $a$--$b$ paths of $G'$ and we are done. So suppose that some path $P \in \mathcal{F}$ contains $u$. Then $P$ is the only path of $G'$ that contains $u$, and $P$ necessarily contains both $y$ and $x$. Let $P'=P[a,y]\odot yuz \odot \overleftarrow{Q_2}\odot P[x,b]$, and note that $P'$ is an $a$--$b$ path that is edge-disjoint from every path in $\mathcal{F}-\{P\}$. So $(\mathcal{F} -\{P\}) \cup \{P'\}$ is a family of $\lambda_G(a,b)$ edge-disjoint $a$--$b$ paths in $G'$, and we conclude that $\lambda_G(a,b)\leq \lambda_{G'}(a,b)$.

To complete the proof of (v), we need to consider the case where $a \in V(G_1)-\{u\}$ and $b \in V(G_2)-\{v\}$. Let $\mathcal{F}$ be a family of $\lambda_G(a,b)$ edge-disjoint $a$--$b$ paths in $G$. If none of these paths contain $u$ or $v$, then $\mathcal{F}$ is a family of $\lambda_{G}(a,b)$ edge-disjoint $a$--$b$ paths in $G'$ and hence $\lambda_{G}(a,b)\leq \lambda_{G'}(a,b)$. Suppose now that $u$ or $v$ appears on some path of $\mathcal{F}$. If $u$ and $v$ both appear in $\mathcal{F}$, then, by Observation \ref{permuting_paths_through_cut-vertex}, we see that $\mathcal{F}$ and $P\in\mathcal{F}$ can be chosen in such a way that $P$ contains both $u$ and $v$.  In this case, let $P'=P[a,y]\odot yuz\odot P[z,b]$. So $(\mathcal{F}-P)\cup \{P'\}$ is a family of $\lambda_G(a,b)$ edge-disjoint $a$--$b$ paths in $G'$.  We assume now that $u$ (but not $v$) lies on $P$. The case where $v$ (but not $u$) lies on $P$ can be argued similarly. Let $R$ be a shortest path in $G_2-v$ from $z$ to a vertex on a path of $\mathcal{F}$, say $R$ is a $z$--$c$ path (note that possibly $c=z$).  So $c$ lies on some path of $\mathcal{F}$.  From the observation made prior to the theorem, we can choose $\mathcal{F}$ and $P\in \mathcal{F}$ in such a way that $P$ contains both $u$ and $c$. Let $P'=P[a,y]\odot yuz\odot R\odot P[c,b]$.  Then  $(\mathcal{F}-\{P\}) \cup \{P'\}$ is a family of $\lambda_G(a,b)$ edge-disjoint $a$--$b$ paths in $G'$. So $\lambda_{G}(a,b)\leq \lambda_{G'}(a,b)$, as desired.

It remains to show that $G'$ is minimally $2$-edge-connected.  The fact that $G'$ is $2$-edge-connected follows immediately from our work above, as $\lambda_{G'}(a,b)\geq \lambda_G(a,b)\geq 2$ for all $a,b\in V(G')$.  By Theorem~\ref{th:verticesdegreegreater2}, every edge of $G$, and hence every edge of $G'$, is incident to a vertex of degree $2$.  It follows directly that $G'$ is minimally $2$-edge-connected.
\end{proof}

We conclude this section by noting that, given a minimally $2$-edge-connected graph $G$ of order $n\geq 5$, for which either the vertices of degree $2$ or the vertices of degree exceeding $2$ are not independent or the graph is not $2$-connected, the proofs of this section implicitly describe an algorithm for constructing a minimally $2$-edge-connected graph $G'$ of the same order $n$ with higher average edge-connectivity than $G$. By repeated application of this algorithm we obtain a $2$-connected minimally $2$-edge-connected graph of order $n$ in which the vertices of degree $2$ and those of degree exceeding $2$ are independent. Moreover, the average edge-connectivity of this graph exceeds that of the other graphs that preceded it in the process.

\subsection{An upper bound on the average edge-connectivity of minimally 2-edge-connected graphs}\label{Bound}
The structural properties proven in Section~\ref{Structure} lead us to a tight upper bound on the average edge-connectivity of a minimally $2$-edge-connected graph.  Both the statement and the proof of this bound are very similar to those of Theorem \ref{KappaBarBound}.
The proof of the edge-analogue of Theorem \ref{KappaBarBound} uses the following two results of Hakimi \cite{Hakimi1962}, and Dankelmann and Oellermann \cite{DankelmannOellermann2005}.

\begin{theorem}[\cite{Hakimi1962}]\label{degree_seq_multigraph}
A sequence $d_1 \ge d_2 \ge \ldots \ge d_n$ of non-negative integers is multigraphical if and only if $\sum_{i=1}^nd_i$ is even and $d_1 \le \sum_{i=2}^nd_i$.
\end{theorem}

\begin{theorem}[\cite{DankelmannOellermann2005}] \label{edge-optimal_degree_seq}
Let $D:d_1 \ge d_2 \ge \ldots \ge d_n$, $n \ge 3$, be a multigraphical sequence with $d_n > 0$ and let $n_1$ denote the number of terms in $D$ that equal $1$.  Then there is an ideally edge-connected multigraph with degree sequence $D$ if and only if
\begin{enumerate}
\item $n_1 \le d_1-d_2$ or
\item $D:n-1,1,1, \ldots, 1$ where $D$ contains $n-1$ terms equal to $1$.
\end{enumerate}
\end{theorem}

In particular, we have the following.

\begin{corollary} \label{edge-optimal}
Let $m$ and $n$ be integers such that $m \ge n \ge 3$. Then there is an ideally edge-connected nearly regular (multi)graph of order $n$ and size $m$.
\end{corollary}

\begin{theorem}\label{LambdaBarBound}
Let $G$ be a minimally $2$-edge-connected graph of order $n$.  Then
\[
\overline{\lambda}(G)\leq 2+\tfrac{(n-2)^2}{4n(n-1)}<\tfrac{9}{4}.
\]
Moreover, let $n=4k+\ell$, where $k,\ell\in\mathbb{Z}$ and $0\leq \ell<k$.
\begin{enumerate}
\item If $k\geq 8$ and $\ell=0$, then
\[
\overline{\lambda}(G)\leq 2+\tfrac{n^2-4n}{4n(n-1)},
\]
with equality if and only if $G$ is obtained from an ideally edge-connected $6$-regular (multi)graph of order $k$ by subdividing every edge.
\item If $k\geq 30$ and $\ell=1$, then
\[
\overline{\lambda}(G)\leq 2+\tfrac{n^2-6n+13}{4n(n-1)},
\]
with equality if and only if $G$ is obtained from an ideally edge-connected nearly regular (multi)graph of order $k$ and size $n-k$ by subdividing every edge.
\item If $k\geq 68$ and $\ell=2$, then
\[
\overline{\lambda}(G)\leq 2+\tfrac{n^2-8n+60}{4n(n-1)},
\]
with equality if and only if $G$ is obtained from an ideally edge-connected nearly regular (multi)graph of either order $k$ and size $n-k$, or order $k+1$ and size $n-k-1$, by subdividing every edge.
\item If $k\geq 30$ and $\ell=3$, then
\[
\overline{\lambda}(G)\leq 2+\tfrac{n^2-6n+17}{4n(n-1)},
\]
with equality if and only if $G$ is obtained from an ideally edge-connected nearly regular (multi)graph of order $k+1$ and size $n-k-1$ by subdividing every edge.
\end{enumerate}
\end{theorem}

\begin{proof}
Let $G$ be an edge-optimal minimally $2$-edge-connected graph of order $n\geq 5$.  By Corollary~\ref{Simple} and Corollary~\ref{edge-optimal_are-bipartite}, $G$ is a simple bipartite graph, with the set of vertices of degree $2$ and the set of vertices of degree exceeding $2$ being independent sets.  The remainder of the proof is analogous to that of Theorem~\ref{KappaBarBound}, with the terminology and notation for connectivity changed to that of edge-connectivity, and Corollary~\ref{edge-optimal} used in place of Theorem~\ref{exists_ideally_connected} to guarantee sharpness.
\end{proof}

\begin{figure}
[htb!]
\begin{minipage}{\textwidth}
\centering{
\raisebox{-0.5\height}{\begin{tikzpicture}
\pgfmathsetmacro{\n}{10}
\pgfmathtruncatemacro{\m}{\n-1}
\pgfmathtruncatemacro{\p}{\m-1}
\foreach \x in {0,...,\m}
{
\vertex (\x) at (-72+360*\x/\n:2cm) {};
}
\foreach \x in {0,...,\p}
{
\pgfmathtruncatemacro{\y}{Mod(\x+1,\n)}
\path
(\x) edge (\y)
(\x) edge[bend right=20] (\y)
(\x) edge[bend left=20] (\y);
}
\node (dots) at ($(\m)!0.5!(0)$) {$\dots$};
\end{tikzpicture}}
\hspace{2cm}
\raisebox{-0.5\height}{\begin{tikzpicture}
\pgfmathsetmacro{\n}{10}
\pgfmathtruncatemacro{\m}{\n-1}
\pgfmathtruncatemacro{\p}{\m-1}
\foreach \x in {0,...,\m}
{
\vertex (\x) at (-72+360*\x/\n:2cm) {};
}
\foreach \x in {0,...,\p}
{
\pgfmathtruncatemacro{\a}{\n+\x}
\pgfmathtruncatemacro{\b}{2*\n+\x}
\pgfmathtruncatemacro{\c}{3*\n+\x}
\hollowvertex (\a) at (-72+180/\n+360*\x/\n:2.3cm) {};
\hollowvertex (\b) at (-72+180/\n+360*\x/\n:2cm) {};
\hollowvertex (\c) at (-72+180/\n+360*\x/\n:1.7cm) {};
}
\foreach \x in {0,...,\p}
{
\pgfmathtruncatemacro{\y}{Mod(\x+1,\n)}
\pgfmathtruncatemacro{\a}{\x+\n}
\pgfmathtruncatemacro{\b}{\x+2*\n}
\pgfmathtruncatemacro{\c}{\x+3*\n}
\path
(\x) edge (\a)
(\a) edge (\y)
(\x) edge (\b)
(\b) edge (\y)
(\x) edge (\c)
(\c) edge (\y);
}
\node (dots) at ($(\m)!0.5!(0)$) {$\dots$};
\end{tikzpicture}}
}
\end{minipage}
\caption{The graph $C_{k}^{[3]}$ (left) and the graph $G_{4k}$ (right) obtained by subdividing every edge of $C_{k}^{[3]}$.  The vertices resulting from subdivision are indicated by hollow circles.}\label{EdgeExample}
\end{figure}
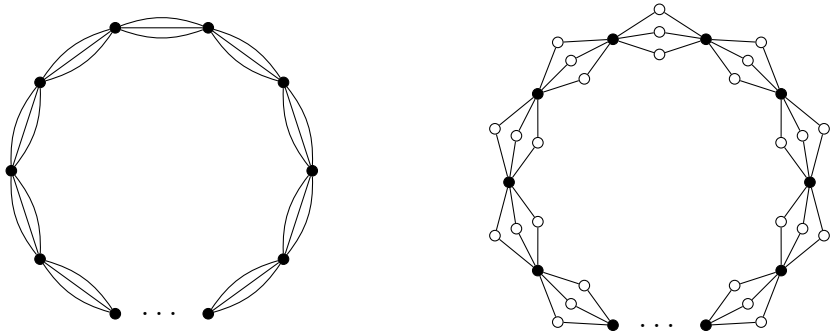

The examples of optimal minimally $2$-connected graphs described at the end of Section~\ref{VertexBoundSection} are now easily seen to be edge-optimal minimally $2$-edge-connected graphs as well.  We can also provide examples of edge-optimal minimally $2$-edge-connected graphs which are not optimal minimally $2$-connected graphs.  For example, for $n=4k$ with $k\geq 8$, let $C_k^{[3]}$ be the graph obtained from $C_k$ by replacing every edge with a bundle of three multiple edges, and let $G_{4k}$ be the graph obtained from $C_k^{[3]}$ by subdividing every edge exactly once (see Figure~\ref{EdgeExample}).  Since $C_k^{[3]}$ is an ideally edge-connected $6$ regular graph on $k$ vertices, we conclude, by Theorem~\ref{LambdaBarBound}, that $G_{4k}$ is an edge-optimal minimally $2$-edge-connected graph.  While $G_{4k}$ is also a minimally $2$-connected graph, note that $C_k^{[3]}$ is clearly not ideally (vertex-)connected, so $G_{4k}$ is not an optimal minimally $2$-connected graph.

\section{Conclusion}
In this paper we obtained sharp bounds for the average connectivity of minimally $2$-connected graphs and the average edge-connectivity of minimally $2$-edge-connected graphs, and we characterized the extremal structures. It remains an open problem to determine an upper bound for the average connectivity of minimally $k$-connected graphs and the average edge-connectivity of minimally $k$-edge-connected graphs for $k \ge 3$.  What can be said about the structure of \emph{optimal} minimally $k$-connected graphs (those with largest average connectivity among all minimally $k$-connected graphs of the same order)?

\begin{conjecture}\label{conjecture}
Let $k\geq 3$, and let $G$ be an optimal minimally $k$-connected graph of order $n$.  Then for $n$ sufficiently large, $G$ is bipartite, with partite sets the set of vertices of degree $k$ and the set of vertices of degree exceeding $k$.
\end{conjecture}

We also conjecture the analogous statement for the edge version.  

\begin{conjecture}\label{conjecture_edge_version}
Let $k\geq 3$, and let $G$ be an edge-optimal minimally $k$-edge-connected graph of order $n$.  Then for $n$ sufficiently large, $G$ is bipartite, with partite sets the set of vertices of degree $k$ and the set of vertices of degree exceeding $k$.
\end{conjecture}

These conjectures are supported by computational evidence for $k=3$ and $k=4$ and $n\leq 11$.  If Conjecture~\ref{conjecture} is true, then for every $k\geq 3$, the proof of the general upper bound of Theorem~\ref{KappaBarBound} generalizes easily to show that $\overline{\kappa}(G)<\tfrac{9}{8}k$ for any minimally $k$-connected graph $G$  of sufficiently large order, depending on $k$.  The edge version is analogous.



\providecommand{\bysame}{\leavevmode\hbox to3em{\hrulefill}\thinspace}
\providecommand{\MR}{\relax\ifhmode\unskip\space\fi MR }
\providecommand{\MRhref}[2]{%
  \href{http://www.ams.org/mathscinet-getitem?mr=#1}{#2}
}
\providecommand{\href}[2]{#2}

\end{document}